\documentclass{paper}
\title{The complex projective plane as a ball quotient}

\RequirePackage{changepage} 

\opext{Prop,prop}
\opext{Isom,Bihol,Deck,Normalizer}
\opext{Ric}
\rmext{sing,norm,lin,orb,even}

\declarecommand\boldsymbol[1]{\textbf{\emph{#1}}}
\declarecommand{\biset}[1]{\mathord{\l\{\mskip-3mu\l\{#1\r\}\mskip-3mu\r\}}}

\begin{document}

\begin{abstract}
    In 1986, Deligne and Mostow constructed a ball quotient $\BB^2 / \Gamma$ biholomorphic to the complex projective plane $\PP^2$ whose branch locus is a line arrangement. In this paper, we show that if $\PP^2$ is realized as a ball quotient whose branch divisor $D$ is an arrangement of smooth pairwise normal-crossing curves, then the orbifold $(\PP^2,D)$ is isomorphic to either the Deligne--Mostow example or a certain degree 9 cover of it. This classification of ``ball quotient structures'' on $\PP^2$ generalizes the $\PP^1$ case due to Poincar\'e.
\end{abstract}

\section{Introduction}
In 1882, Poincar\'e classified the ball quotient structures on $\PP^1$ in the following sense: If a quotient $\BB^1 / \Gamma$ of the 1-dimensional open unit ball $\BB^1$ by a lattice $\Gamma$ is biholomorphic to the complex projective line $\PP^1$, then the branch indices $(b_1,\dots,b_k)$ of the quotient map satisfy
\begin{gather}
\label{intro:P1}
    \sum_{i=1}^k \l( 1 - \frac{1}{b_i} \r) > 2,
\end{gather}
and conversely, every tuple $(b_1,\dots,b_k)$ satisfying \eqref{intro:P1} is realizable in this manner. This result is implicit in Poincar\'e's original paper \cite{poincare-1882-tgf} on elliptic elements in Fuchsian groups.

This paper considers the analogous question in dimension 2:
\vspace{-0.5ex}
\begin{adjustwidth}{0.96cm}{0.96cm}
\it
If $\PP^2$ is realized as a quotient $\BB^2 / \Gamma$ then which divisors $D = \sum_i \big( 1 - \frac{1}{b_i} \big) D_i$ on $\PP^2$ can occur \\[-0.4ex]
as the branch divisor of the quotient map $\BB^2 \to \PP^2$?
\end{adjustwidth}
\medskip

A \emph{ball quotient} is a quotient $X = \BB^n / \Gamma$ of the open unit ball $\BB^n \subset \CC^n$ by a (not necessarily torsion-free) lattice $\Gamma < \Aut \BB^n \iso \PU_{1,n}(\CC)$. If $\Gamma$ has torsion then $X$ has the structure of a complex-hyperbolic orbifold $(X,D)$ whose orbifold divisor $D$ is given by the branch divisor of the quotient map $\BB^n \to X$. A divisor $D$ arising in this manner is called a \emph{ball quotient divisor} on $X$ and the pair $(X,D)$ is called a \emph{ball quotient orbifold}.

\paragraph{Two examples.}
It is not obvious \emph{a priori} whether $\PP^2$ admits a ball quotient divisor at all. In 1986, Deligne--Mostow \cite{deligne-mostow-1986-mhf} constructed a family of ball quotients arising as certain stable moduli spaces, with the objective of constructing new non-arithmetic lattices in $\PU_{1,n}(\CC)$. The Deligne--Mostow ball quotient with weights $\big( \frac 56, \frac{5}{12}, \frac 14, \frac 14, \frac 14 \big)$ is biholomorphic to $\PP^2$ and the branch divisor $D_Q$ of the quotient map $\BB^2 \to \PP^2$ is the \emph{complete quadrilateral}
\begin{gather*}
    XYZ(X-Y)(X-Z)(Y-Z) = 0
\end{gather*}
with weight 3 on lines $X$, $Y$, $Z$ and weight 2 on the remaining lines (depicted in \cref{fig:line-arrangements}, right). This line arrangement is the unique (up to projective automorphisms) arrangement of 6 lines with 4 triple points, since $\PGL_3(\CC)$ acts transitively on ordered 4-tuples of points with no three collinear. For a direct proof that $D_Q$ is a ball quotient divisor, see \cref{P2:DQ-DH-are-ball-quotients}.
\medskip

A second example of a ball quotient structure on $\PP^2$ is produced by taking a certain finite orbifold cover of $(\PP^2, D_Q)$: The degree 9 map $f_3: \PP^2 \to \PP^2$ given by
\begin{align*}
    (X:Y:Z) &\mapsto (X^3:Y^3:Z^3)
\end{align*}
is a normal branched cover of $\PP^2$ whose branch locus is the triangle $XYZ = 0$ with all branch indices equal to 3. This triangle makes up 3 of the 6 lines in the complete quadrilateral $D_Q$; the preimage of the remaining 3 lines under $f_3$ is the \emph{dual Hesse arrangement}
\begin{gather*}
    (X^3 - Y^3)(X^3 - Z^3)(Y^3 - Z^3) = 0
\end{gather*}
(see \cref{fig:line-arrangements} for a partial illustration). This is the unique (up to projective automorphisms) arrangement of 9 lines in $\PP^2$ with 12 triple points \cite[Theorem 1]{lampa-baczynska-wojcik-2019-dha}. The weighted arrangement $D_H$ given by the dual Hesse arrangement with all lines given weight 2 turns out to be a ball quotient divisor (proven in \cref{P2:DQ-DH-are-ball-quotients}, see also \cref{covering-ball-quotient}).

{ \centering
\begin{tabular}{ C{0.3\textwidth} C{0.1\textwidth} C{0.3\textwidth} }
\img[0.3]{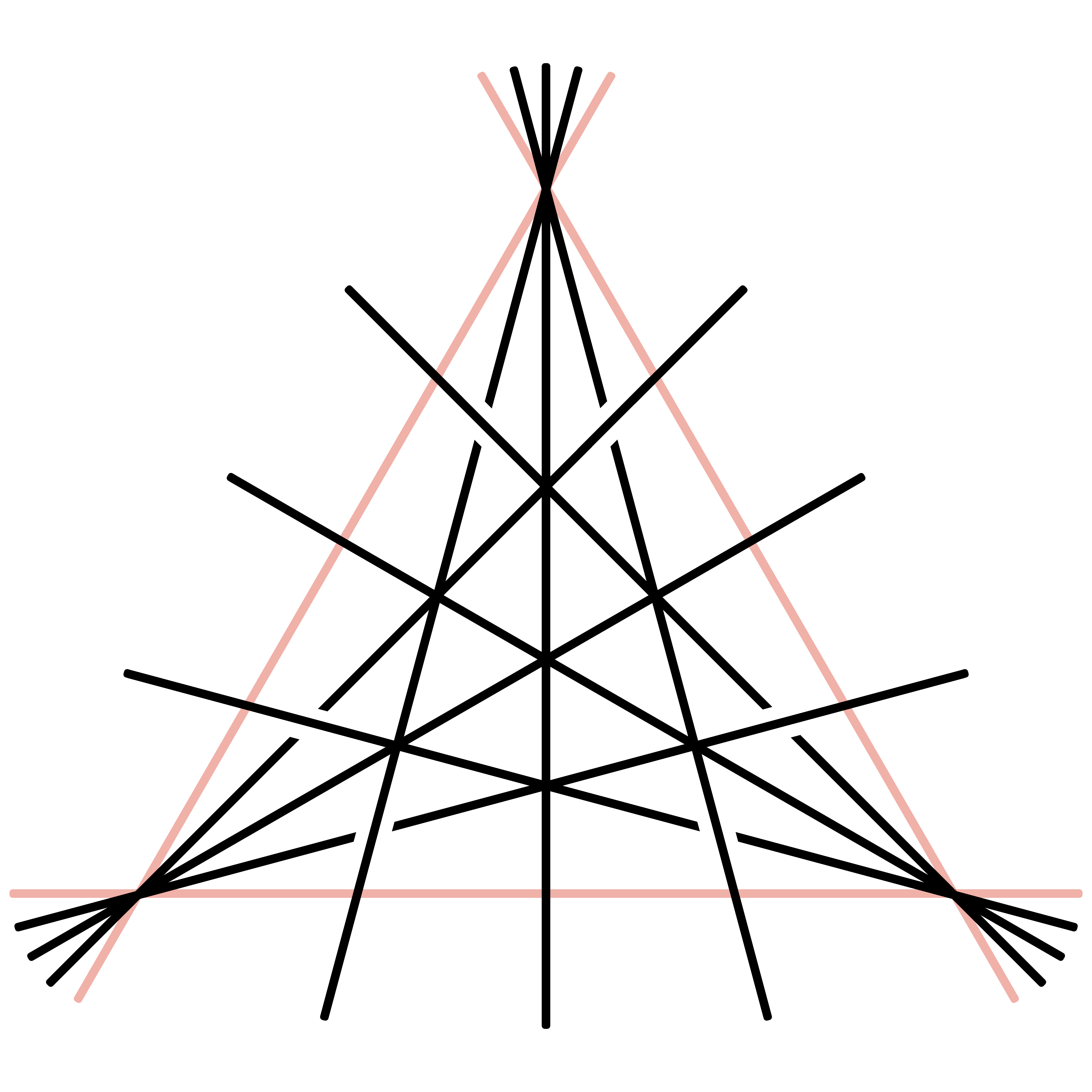}
& ${\Large \xlongrightarrow {\quad f_3 \quad}}$
& \img[0.3]{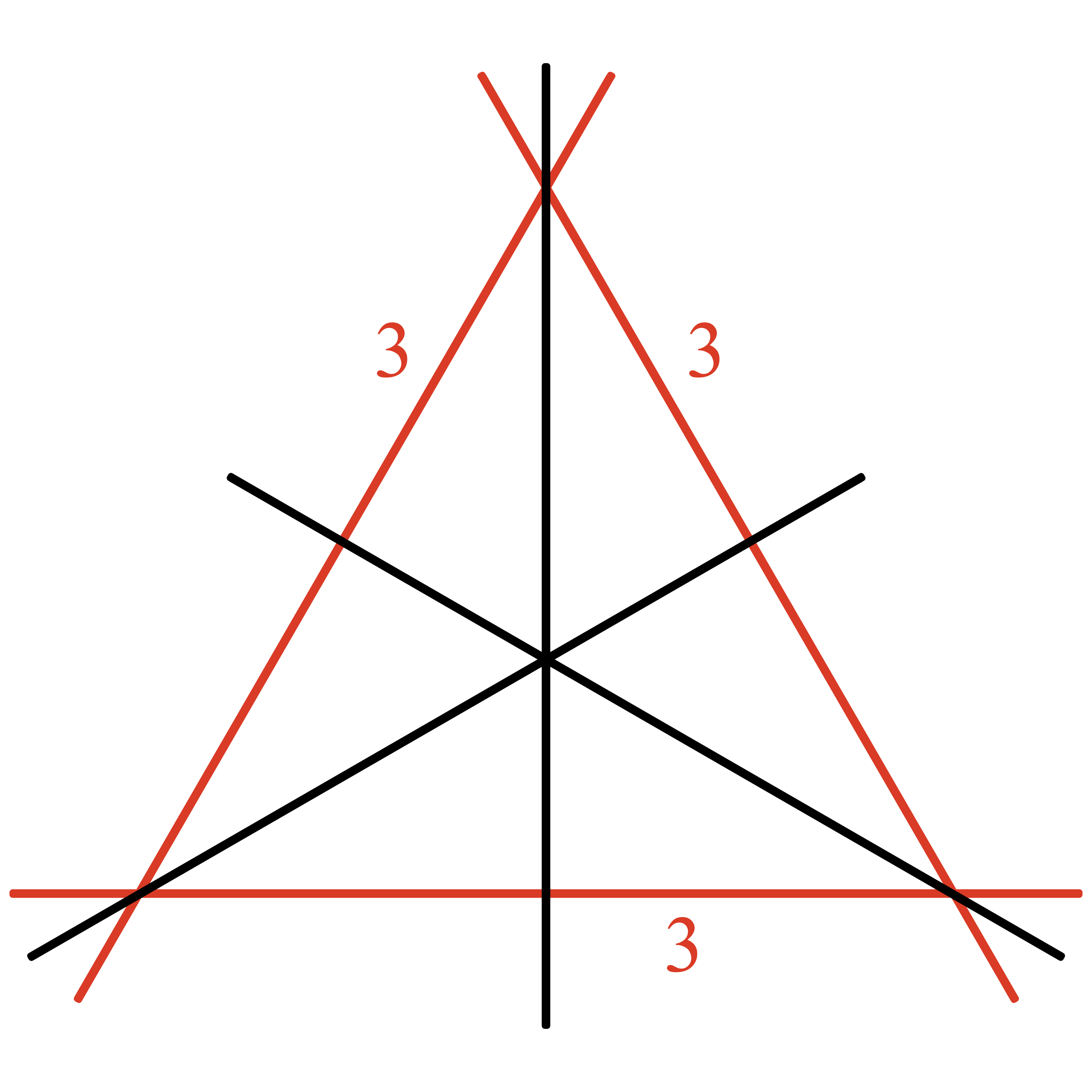} \\
$D_H$ & & $D_Q$
\end{tabular}
\captionsetup{width=0.88\textwidth}
\captionof{figure}{Illustration of the orbifold divisors $D_H$ and $D_Q$ on $\PP^2$. Weight-2 lines have been left unlabelled. The dual Hesse arrangement $D_H$ is represented by the 9 black lines on the left (the faint red lines are not part of $D_H$). Note that this is only a partial illustration, as two of the 12 triple points are not pictured---the dual Hesse arrangement cannot be realized by a real line arrangement, so no illustration is completely faithful. The complete quadrilateral $D_Q$ is pictured on the right, consisting of the black lines \emph{and} the red lines for 6 lines in total. The red triangles denote the ramification and branch loci of the degree 9 map $f_3: \PP^2 \to \PP^2$. Each black line on the left is mapped 3-to-1 onto one of the black lines on the right by $f_3$.}
\label{fig:line-arrangements}
}
\medskip

The main result of this paper is the following:

\begin{theorem}
\label{main-theorem}
Let $D$ be an orbifold divisor on $\PP^2$ whose irreducible components are smooth and pairwise normal-crossing. Then $(\PP^2,D)$ is a ball quotient orbifold if and only if $D$ is (projectively equivalent to) $D_Q$ or $D_H$ as described above.
\end{theorem}

This paper provides a direct proof that $D_Q$ and $D_H$ are ball quotient divisors (\cref{P2:DQ-DH-are-ball-quotients}) using a deep uniformization result of Kobayashi--Nakamura--Sakai. This part of \cref{main-theorem} is certainly not new: In their paper \cite{deligne-mostow-1986-mhf}, Deligne--Mostow showed that $D_Q$ is a ball quotient divisor, and it was known to Thomas H\"ofer \cite{hofer-1985-bav} that both are ball quotient divisors. That $D_Q$ and $D_H$ are the \emph{only} line arrangements (or smooth pairwise normal-crossing arrangements) that arise as ball quotient divisors appears to be new.

\begin{remarks}[``Pairwise normal-crossing'']
\item
    The hypothesis that the components of $D$ are smooth and pairwise normal-crossing is strictly weaker than assuming that $D$ is a smooth normal-crossing divisor. Neither $D_Q$ nor $D_H$ is normal-crossing but both are pairwise normal-crossing since any pair of lines on $\PP^2$, viewed in isolation from all the other lines, meet in a normal crossing. Indeed, a consequence of \cref{main-theorem} is that no smooth normal-crossing divisor $D$ on $\PP^2$ is a ball quotient divisor.
\item \begin{minipage}[t]{0.725\linewidth}
    There exist ball quotient divisors on $\PP^2$ whose components are smooth but not pairwise normal-crossing. Uluda\u g showed \cite[Proposition 1(iii)]{uludag-2004-crb} that the following weighted arrangement (pictured in \cref{fig:uludag}) is a ball quotient divisor: a conic with weight 3, and three distinct lines tangent to the conic with weights 3, 4, and 4. This arrangement has singularities
    \begin{gather*}
        [3,4]^2\, [4,4]\, [3(2)3]\, [3(2)4]^2
    \end{gather*}
    (see \cref{reflection-group-singularities} and \cref{singularity-notation} for notation). The methods of this paper fail to account for orbifold divisors with singularities of type $[a(m)b]$.
    \end{minipage}
    \hfill
    \begin{minipage}[t]{0.24\linewidth}
        \vspace{-3.5ex} 
        \img[1]{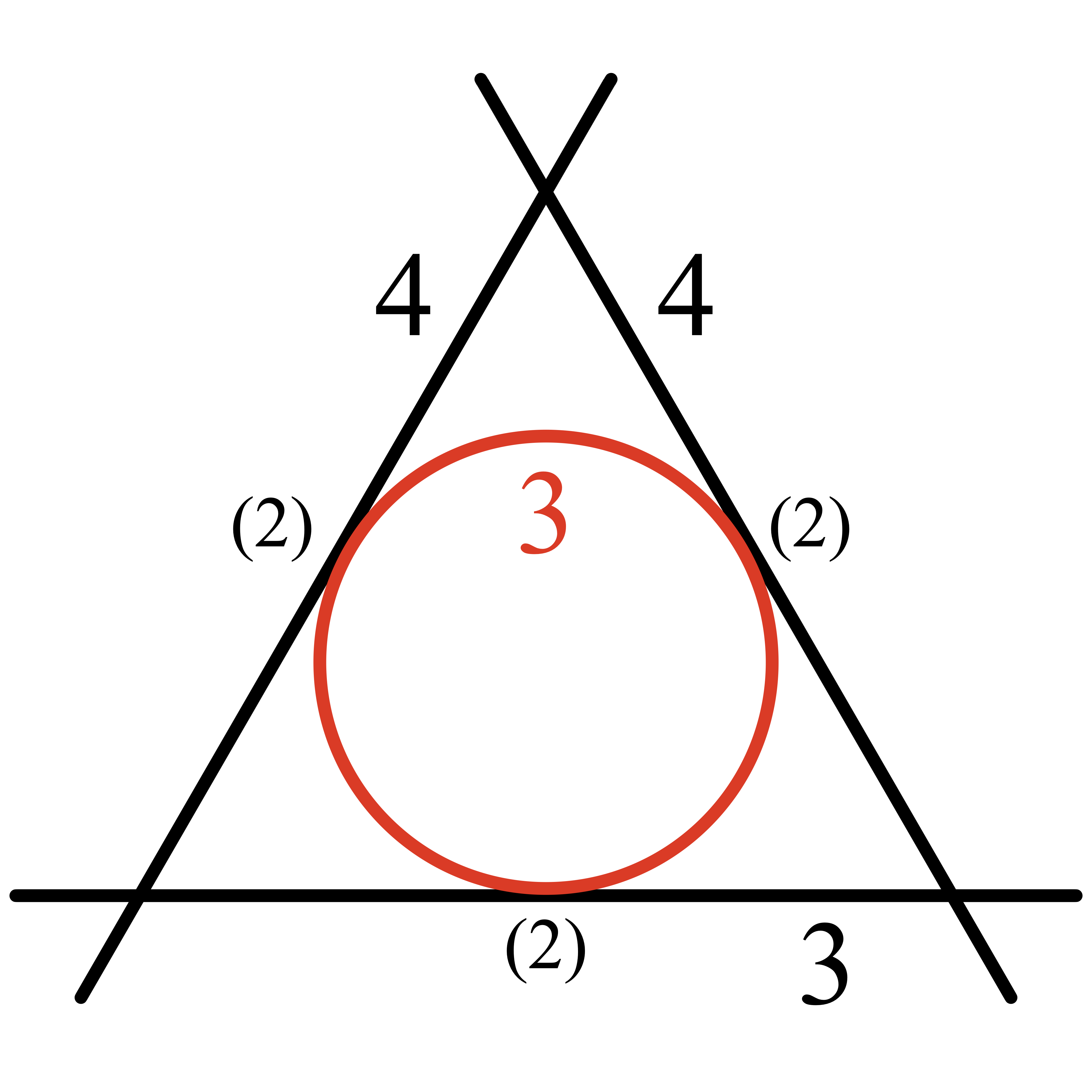}
        \vspace{-4.5ex} 
        \captionsetup{width=0.97\textwidth}
        \captionof{figure}{A ball quotient divisor on $\PP^2$ \emph{not} within the scope of  \cref{main-theorem}.\label{fig:uludag}}
    \end{minipage}
\end{remarks}
\vspace{-1.2ex} 

\paragraph{Methods.}
The main tools used to rule out orbifold divisors are Hirzebruch proportionality (\cref{hirzebruch-proportionality}) and Enoki's version of proportionality for curves (\cref{relative-proportionality}), which state that compact ball quotients satisfy strict numerical conditions. These conditions motivate the introduction of certain numerical invariants, defined for all compact smooth orbifolds $(X,D)$, that vanish for ball quotient orbifolds (see \cref{orbifold-prop-definitions} and \cref{constraints}).

Then the idea is simple in theory: Derive explicit formulas for the orbifold invariants, compute invariants for all possible orbifold divisors, and check for vanishing. In practice, the full class of orbifold divisors on a surface---even one as simple as $\PP^2$---is very large and difficult to handle in its entirety. A major hurdle is that the invariants depend on the configuration of singularities, so linear equivalence is not easily exploited.

The strategy in this paper is to split each invariant into (1) the part that is preserved by linear equivalence of $D$, and (2) the part that depends on the singularities on $D$. It turns out that this second part can always be expressed as the total contribution from all individual singular points on $D$, where the contribution from a singular point $x \in D_\sing$ is a number depending only on the local picture near $x$.

The theory of complex reflection groups plays a key role in the study of such singularities: By a theorem of Chevalley, the local group associated to any singularity of the orbifold divisor of a smooth orbifold is a complex reflection group. Thus the Shephard--Todd classification gives a complete classification of possible singularities on an orbifold divisor on a smooth surface (\cref{reflection-group-singularities}).

To verify that the examples $D_Q$ and $D_H$ are indeed ball quotient divisors, the main tool is the orbifold version of Yau's uniformization theorem, due to Kobayashi--Nakamura--Sakai (\cref{ball-quotients:orbifold}).

\paragraph{Historical context.}
Hirzebruch had the following idea for producing ball quotients systematically by branched covering constructions: Take a line arrangement on $\PP^2$, blow up every point with multiplicity $\geq 3$. The strict transform of the line arrangement, combined with all exceptional divisors, is a normal-crossing divisor on a blow-up of $\PP^2$, to which weights may be assigned. Then construct an abelian cover of this surface and compute the Chern numbers of the cover. Hirzebruch's original paper \cite{hirzebruch-1983-ala} found ball quotient divisors on various blow-ups of $\PP^2$.

Subsequent papers of Hirzebruch---and later, his students H\"ofer and Tretkoff---pushed this idea further, producing many examples of ball quotient divisors on blow-ups of $\PP^2$ \cite{hirzebruch-1985-ase,hofer-1985-bav,barthel-ea-1987-gua,tretkoff-2016-cbq} as well as blow-ups of abelian surfaces \cite{hirzebruch-1984-cna}. Stover \cite{stover-2023-pcb} later produced ball quotient divisors on blow-ups of products of curves with the same approach.

It should be noted that the ball quotient orbifolds studied by Hirzebruch are generally non-minimal and have normal-crossing ball quotient divisors, whereas the present paper is focused on ball quotient structures on $\PP^2$ \emph{without blowing up any points}.

\paragraph{Future directions.}
Using the computational framework set up in \cref{sec:computational-framework}, results similar to \cref{main-theorem} can be obtained for compact surfaces other than $\PP^2$. This will be addressed in a sequel. A complete classification of ball quotient structures on $\PP^2$ (without the smooth and pairwise normal-crossing hypothesis) would be satisfying, but would likely require new insights to handle other singularities.

\paragraph{Structure of paper.} \cref{sec:background} fixes notation and reviews facts about finite group actions, orbifolds, and complex reflection groups. In \cref{sec:computational-framework}, ball quotients are introduced and the main computational framework (\cref{ball-quotients:orbifold-ample,constraints,ball-quotients:orbifold}) is set up. \cref{sec:main-theorem-proof} is devoted to the proof of \cref{main-theorem}.

\paragraph{Statement on AI usage.} AI tools were not involved in this project in any way, with the exception of copyediting at the final stage. No part of this paper was written by AI.

\paragraph{Acknowledgements.} I thank my advisor Benson Farb for many useful mathematical conversations, for comments on earlier drafts of this paper, and for suggesting this project in the first place. I am grateful for his endless patience and support. I would also like to thank Eduard Looijenga, Matthew Stover, Zhiwei Zheng, and Gregorio Baldi, all of whom shared valuable insights.

\tableofcontents

\section{Background}
\label{sec:background}

This section serves to collect standard background material and to fix notation and terminology.
\vspace{-0.5ex} 

\paragraph{Conventions.} All varieties are over $\CC$. The preferred category is that of (normal) analytic varieties---the notation $\Aut X$ is reserved for the biholomorphism group of $X$. In this paper, a \emph{surface} is a 2-dimensional quasi-projective algebraic variety over $\CC$.

\paragraph{Weighted arrangements, weighted singularities.}
Let $X$ be a normal variety. A \emph{weighted $\QQ$-divisor} on $X$ means a $\QQ$-divisor of the form
\begin{gather*}
    D = \sum_i \l( 1 - \frac{1}{b_i} \r) D_i,
    \qquad b_i \in \ZZ_{\geq 2}.
\end{gather*}
\vspace{-2.7ex} 

Such a divisor should be viewed as a weighted arrangement of irreducible subvarieties $\{D_i\}_i$ where the \emph{weight} of a component $D_i$ is the integer $b_i$.
\begin{itemize}
\item
A (\emph{weighted}) \emph{singularity} is a singular point on a weighted $\QQ$-divisor $D$, considered up to holomorphic local coordinate change preserving the weights of the components of $D$ meeting the singularity.
\item
A \emph{singularity pair} is a pair $(x,C)$ where $x$ is a weighted singularity and $C$ is a (possibly reducible) codimension 1 germ at $x$, considered up to holomorphic local coordinate change preserving weights and preserving the distinguished germ.
\end{itemize}

The primary examples of weighted $\QQ$-divisors in the context of this paper are branch divisors (of quotient maps) and orbifold divisors, where the weights are given by the branch indices. In both cases the terms \emph{weight} and \emph{branch index} will be used interchangeably.

\subsection{Quotient spaces, ramification and branching}
In preparation for working with orbifolds, this subsection fixes terminology related to quotient maps and branch data, and collects several lemmas.

Let $Y$ be a complex manifold and $\Gamma$ a group of biholomorphisms acting properly discontinuously with quotient space $X \defeq Y / \Gamma$ and let $\pi: Y \to X$ denote the quotient map. If $\Gamma$ does not act freely, then the quotient space comes equipped with branching data that is often useful to remember.

\begin{itemize}[itemsep=0pt,parsep=7pt,topsep=0pt] 
\item
    The \emph{ramification locus} of $\pi$ is the set
    \begin{gather*}
        R \defeq \{y \in Y : \stab_\Gamma(y) \neq 1\} \subset Y.
    \end{gather*}
    The ramification locus is a union of complex submanifolds (see \cref{local-model-reflection-group}). If $R_j$ is an irreducible codimension 1 component of the ramification locus, then the pointwise stabilizer $\Gamma_{R_j}$ is cyclic and the \emph{ramification index} of $R_j$ is defined to be $k_j \defeq |\Gamma_{R_j}|$.

\item
    For $x \in X$, the \emph{(quotient) local group} at $x$ is (the abstract isomorphism class of)
    \begin{gather*}
        \Gamma_x \defeq \stab_\Gamma(\tilde x)
        \quad\text{for any $\tilde x \in \pi^{-1}(x) \subset Y$}.
    \end{gather*}
    This defines a conjugacy class of subgroups of $\Gamma$.

\item
    The \emph{branch locus} of $\pi$ is the set
    \begin{gather*}
        B \defeq \{x \in X : \Gamma_x \neq 1\} = \pi(R) \subset X.
    \end{gather*}
    The irreducible components of the branch locus are analytic subvarieties of $X$.

\item
    The \emph{branch divisor} of $\pi$ is the weighted $\QQ$-divisor
    \begin{gather*}
        D \defeq \sum_i \l( 1 - \frac{1}{b_i} \r) D_i \ \in \Div_\QQ X
    \end{gather*}
    \vspace{-3ex} 

    where $\{D_i\}_i$ are the irreducible codimension 1 components of the branch locus and the \emph{branch index} $b_i$ of an irreducible component $D_i$ is
    \begin{alignat*}{2}
        b_i &\defeq k_j \quad &&\text{for any $R_j$ such that $\pi(R_j) = D_i$} \\
        &= |\Gamma_x| \quad &&\text{for any $x \in D_i - D_\sing$}.
    \end{alignat*}
\end{itemize}

The branch divisor has the following useful property:

\begin{lemma}[e.g. {\cite[Section I.16]{barth-ea-1984-ccs}}] 
\label{lift-canonical}
    Let $Y$ be a complex manifold and $G < \Aut Y$ a finite group with quotient $X \defeq Y / G$. Let $D$ denote the branch divisor of the quotient map $\pi: Y \to X$. The canonical divisors of $X$ and $Y$ are related by
    \begin{gather*}
        K_Y = \pi^*(K_X + D).
    \end{gather*}
\end{lemma}

The following lemma is originally due to Cartan:

\begin{lemma}
\label{finite-group-action-linearization}
    Let $Y$ be a complex manifold and let $G < \Aut Y$ be a finite group acting on $Y$. If $y$ is a global fixed point of $G$ then there exists an open neighbourhood $V$ of $y$ preserved by $G$ and local coordinates on $V$ centred at $y$ with respect to which $G$ acts by complex-linear automorphisms. 
\end{lemma}

\begin{proof}
    Let $V_0$ be an open neighbourhood of $y$ with a coordinate map
    \begin{gather*}
        \varphi: V_0 \xlongrightarrow\iso U_0 \overset{\text{open}}{\subseteq} \CC^n
    \end{gather*}
    where $\varphi(y) = 0$. Assume that $V_0$ is $G$-invariant, after replacing $V_0$ with the finite intersection $\bigcap_{g\in G} gV_0$ (this is non-empty because every element of $G$ fixes $y$). For each $g \in G$, view $g$ as a biholomorphism of $U_0 \subseteq \CC^n$ and view the differential map $d_0g: T_0 \CC^n \to T_0 \CC^n$ as a linear automorphism of $\CC^n$ after making the canonical identification $T_0 \CC^n = \CC^n$. Define the holomorphic map
    \begin{gather*}
        \alpha \defeq \frac{1}{|G|} \sum_{g \in G} (d_0g)^{-1} g : U_0 \to \CC^n.
    \end{gather*}
    Observe that $d_0\alpha : T_0 \CC^n \to T_0 \CC^n$ is the identity map so $\alpha$ is injective on some neighbourhood $U$ of $0$. By replacing $U$ with the finite intersection $\bigcap_{g \in G} gU$, it can be assumed that $U$ is $G$-invariant. Let $V = \varphi^{-1}(U)$ be given the coordinate map
    \begin{gather*}
        \alpha\varphi: V \xlongrightarrow\iso \alpha(U) \subset \CC^n.
    \end{gather*}

    For $h \in G$, the action of $h$ on $\alpha(U)$ is given by the biholomorphism $\alpha h \alpha^{-1}$. By definition of $\alpha$,
    \begin{gather*}
        d_0 h \alpha
        = \frac{1}{|G|} \sum_{g \in G} d_0h(d_0g)^{-1} g
        = \frac{1}{|G|} \sum_{g \in G} d_0(hg^{-1}) g
        = \frac{1}{|G|} \sum_{g \in G} d_0g^{-1} gh = \alpha h.
    \end{gather*}
    Therefore $\alpha h \alpha^{-1} = d_0 h$ is linear.
\end{proof}

\begin{lemma}
\label{local-model-reflection-group}
Let $Y$ be a complex manifold and let $\Gamma < \Aut Y$ act properly discontinuously with quotient space $X \defeq Y / \Gamma$. For each $y \in Y$, the stabilizer $\stab_\Gamma(y)$ is finite and there exists a coordinate chart $V \subset Y$ of $y$ preserved by $\stab_\Gamma(y)$ such that
\begin{enumerate}[label={\normalfont(\arabic*)}]
\item
the action of $\stab_\Gamma(y)$ with respect to the local coordinates is by complex-linear automorphisms; and

\item
the restriction
\begin{gather*}
    \pi|_V: V \to \pi(V)
\end{gather*}
is the quotient of $V$ by the action of $\stab_\Gamma(y)$.
\end{enumerate}
\end{lemma}

\begin{proof}
    Since $\Gamma$ acts properly discontinuously, any point $y \in Y$ admits an open neighbourhood $V$ such that
    \begin{gather*}
        S \defeq \{g \in \Gamma : gV \cap V \neq \emptyset\} \text{ is finite.}
    \end{gather*}
    Then $\stab_\Gamma(y)$ is finite since $\stab_\Gamma(y) \subseteq S$. Assume, by intersecting with an appropriate chart given by \cref{finite-group-action-linearization}, that $V$ admits a local coordinate system centred at $y$ with respect to which the action of $\stab_\Gamma(y)$ is complex-linear.

    If $g \in S$ but $g \cdot y \neq y$ then there exist disjoint open sets $A$ and $B$ containing $y$ and $g \cdot y$ respectively since $Y$ is Hausdorff. The open set
    \begin{gather*}
        V' \defeq A \cap g^{-1}(B) \cap V
    \end{gather*}
    contains $y$ and $g V' \cap V' = \emptyset$ because $gV' \subset B$ and $V' \subset A$. By replacing $V$ with $V'$ in this manner finitely many times, it can be assumed that $S = \stab_\Gamma(y)$. Finally, to ensure that $V$ is preserved by $\stab_\Gamma(y)$, replace $V$ with the open subset
    \begin{gather*}
        \bigcap_{g \in \stab_\Gamma(y)} gV.
    \end{gather*}
    Then $V$ satisfies
    \begin{gather*}
    \begin{cases}
        gV = V & g \in \stab_\Gamma(y) \\
        gV \cap V = \emptyset & g \notin \stab_\Gamma(y)
    \end{cases}
    \end{gather*}
    so the restriction $\pi|_V: V \to \pi(V)$ is the quotient map of $V$ by $\stab_\Gamma(y)$.
\end{proof}

\subsection{Orbifolds}
\label{sec:orbifolds}
The notion of an orbifold is meant to describe a space locally modelled by finite quotients of $\RR^n$. The language of orbifolds is not at all standardized; the aim of this subsection is to precisely define the orbifold vocabulary required for this paper, with minimal technical detail. Refer to e.g. \cite{caramello-2022-io} for a rigorous and far more general treatment of the algebraic topology of orbifolds. Aspects of orbifold theory specific to the uniformization of complex algebraic surfaces may be found in \cite{kobayashi-1990-ucs} or \cite{uludag-2007-otu}. 

Let $X$ be a normal complex variety. A \emph{complex orbifold structure} on $X$ is a maximal open cover by charts with biholomorphisms to finite quotients of $\CC^n$ subject to certain compatibility conditions on the chart group actions (the precise definition may be found in \cite{caramello-2022-io}).
\begin{itemize}
\item
    The \emph{orbifold local group} $\Gamma_x$ at a point $x$ is the quotient local group of $x$ in any chart. As an abstract group, this does not depend on the chart.
\item
    The \emph{orbifold locus} is the set
    \begin{gather*}
        \{x \in X : \Gamma_x \neq 1 \} \subset X,
    \end{gather*}
    or equivalently the union of the branch loci across all of the charts.
\item
    The \emph{orbifold divisor} (of the orbifold structure) is the weighted $\QQ$-divisor
    \begin{gather*}
        D \defeq \sum_i \l( 1 - \frac{1}{b_i} \r) D_i \ \in \Div_\QQ X
    \end{gather*}
    where $\{D_i\}_i$ are the codimension 1 components of the orbifold locus, and the branch index $b_i$ of a component $D_i$ is the branch index of $D_i$ in any chart intersecting $D_i$ non-trivially. This does not depend on the choice of chart.
\end{itemize}
The primary example of a complex orbifold is a quotient $Y / \Gamma$ where $Y$ is a smooth surface and $\Gamma$ is a group acting properly discontinuously on $Y$. An orbifold structure on $X$ arising from a quotient map is called \emph{uniformizable} (or \emph{good}), and the quotient map $Y \to X$ (or sometimes just the manifold $Y$) is called a \emph{uniformization} of the orbifold.
\medskip

In this paper, all orbifolds will have underlying space a smooth surface $X$. In this setting, it is convenient to use the following equivalent definitions:
\begin{itemize}
\item
    An \emph{orbifold divisor} on a smooth surface $X$ is a weighted $\QQ$-divisor having only reflection group singularities (defined in \cref{sec:reflection-groups}). Such a divisor $D$ on $X$ specifies a unique orbifold structure on $X$ whose orbifold divisor is $D$ because the orbifold local group and local chart of a singular point can be recovered from the weighted singularity type (see \cref{reflection-group-singularities}).
\item
    An \emph{orbifold} (\emph{surface}) is a pair $(X,D)$ where $X$ is a smooth surface and $D$ is an orbifold divisor. An orbifold will always be denoted as a pair $(X,D)$ (as opposed to $X$, which will be reserved for the underlying manifold).
\end{itemize}

\emph{Shorthand.}\enspace
Specifically in this paper, the term \emph{orbifold} implicitly means an orbifold surface (i.e. all orbifolds are assumed to have underlying space a smooth surface). Adjectives of manifolds (e.g. \emph{smooth}, \emph{compact}) applied to an orbifold $(X,D)$ are understood to describe the underlying surface $X$.

\begin{remark}
\label{covering-ball-quotient}
    There is a notion of an ``orbifold covering map'' (see \cite[Section 2.3]{caramello-2022-io}), generalizing both an unbranched covering map and a uniformization map, with the property that any orbifold cover of a ball quotient orbifold is again a ball quotient orbifold. This motivates the construction of $D_H$ from the introduction, since $f_3: (\PP^2, D_H) \to (\PP^2, D_Q)$ indeed defines a normal orbifold cover, and the fact that $(\PP^2, D_H)$ is a ball quotient orbifold can be deduced from Deligne--Mostow's proof that $(\PP^2, D_Q)$ is a ball quotient orbifold. A consequence is that the lattices $\Gamma_H$ and $\Gamma_Q$ associated to the ball quotient orbifolds $(\PP^2, D_H)$ and $(\PP^2,D_Q)$ respectively are commensurable (in fact, $\Gamma_H$ is a finite-index subgroup of $\Gamma_Q$). Since $\Gamma_Q$ is arithmetic \cite[p. 86]{deligne-mostow-1986-mhf}, it follows that $\Gamma_H$ is arithmetic. The definition of ``orbifold covering map'' will not be used in this paper outside of this remark.
\end{remark}

\subsection{Complex reflection groups, reflection group singularities}
\label{sec:reflection-groups}

The objective of this subsection is to define \emph{reflection group singularities}, explain their relevance to the study of smooth orbifolds, and give a complete classification of reflection group singularities in dimension 2 (summarized in \cref{reflection-group-singularities}). For detailed references, see the notes following \cref{reflection-group-singularities}.
\medskip

A \emph{complex reflection} (or \emph{pseudo-reflection}) is a finite-order linear automorphism of $\CC^n$ whose fixed set is a codimension 1 subspace. A \emph{complex reflection group} is a subgroup of $\GL_n(\CC)$ generated by reflections.

By a theorem of Chevalley, the quotient of $\CC^n$ by the action of a finite complex reflection group $G$ is biregular to $\CC^n$. The branch divisor of the quotient map $\CC^n \to \CC^n / G \iso \CC^n$ is singular at the origin unless $G$ is generated by a single reflection. A weighted singularity arising as the origin singularity for some reflection group is called a \emph{reflection group singularity}.

It turns out that finite complex reflection groups are characterized by having quotient biregular to $\CC^n$.

\begin{theorem}[Chevalley \cite{chevalley-1955-ifg}, Shephard--Todd \cite{shephard-todd-1954-fur}]
\label{chevalley-smooth-quotient}
A finite group $G \subset \GL_n(\CC)$ is a reflection group if and only if the quotient space $\CC^n / G$ is isomorphic to $\CC^n$ as an affine algebraic variety.
\end{theorem}

\begin{corollary}\ 
\begin{parts}
\item
    If $Y$ is a complex manifold and $\Gamma < \Aut Y$ acts properly discontinuously with smooth quotient $X \defeq Y / \Gamma$, then every singularity on the branch divisor of the quotient map $Y \to X$ is a reflection group singularity.
\item
    If $X$ is a smooth variety with an orbifold structure, then any singularity on the orbifold divisor is a reflection group singularity.
\end{parts}
\end{corollary}

A complete classification of finite complex reflection groups in all dimensions was given by Shephard--Todd \cite{shephard-todd-1954-fur}. The branch divisor of any complex reflection group can be computed from the data given in the classification, so reflection group singularities are also classified. The reflection group singularities in dimension 2 are enumerated in \cref{reflection-group-singularities}.

\begin{longtblr}[
    caption={Classification of reflection group singularities in dimension 2. For references and notation, see the remarks immediately below. },
    label={reflection-group-singularities}
]{
  colspec={ c | l l | l | l },
}
\SetCell{l}
\text{local model}
    & \SetCell[c=2]{m} \text{singularity type} &
    & \text{S-T}
    & \text{group structure} \\
\SetCell[r=2]{t,c,mode=text}
{\img[0.22]{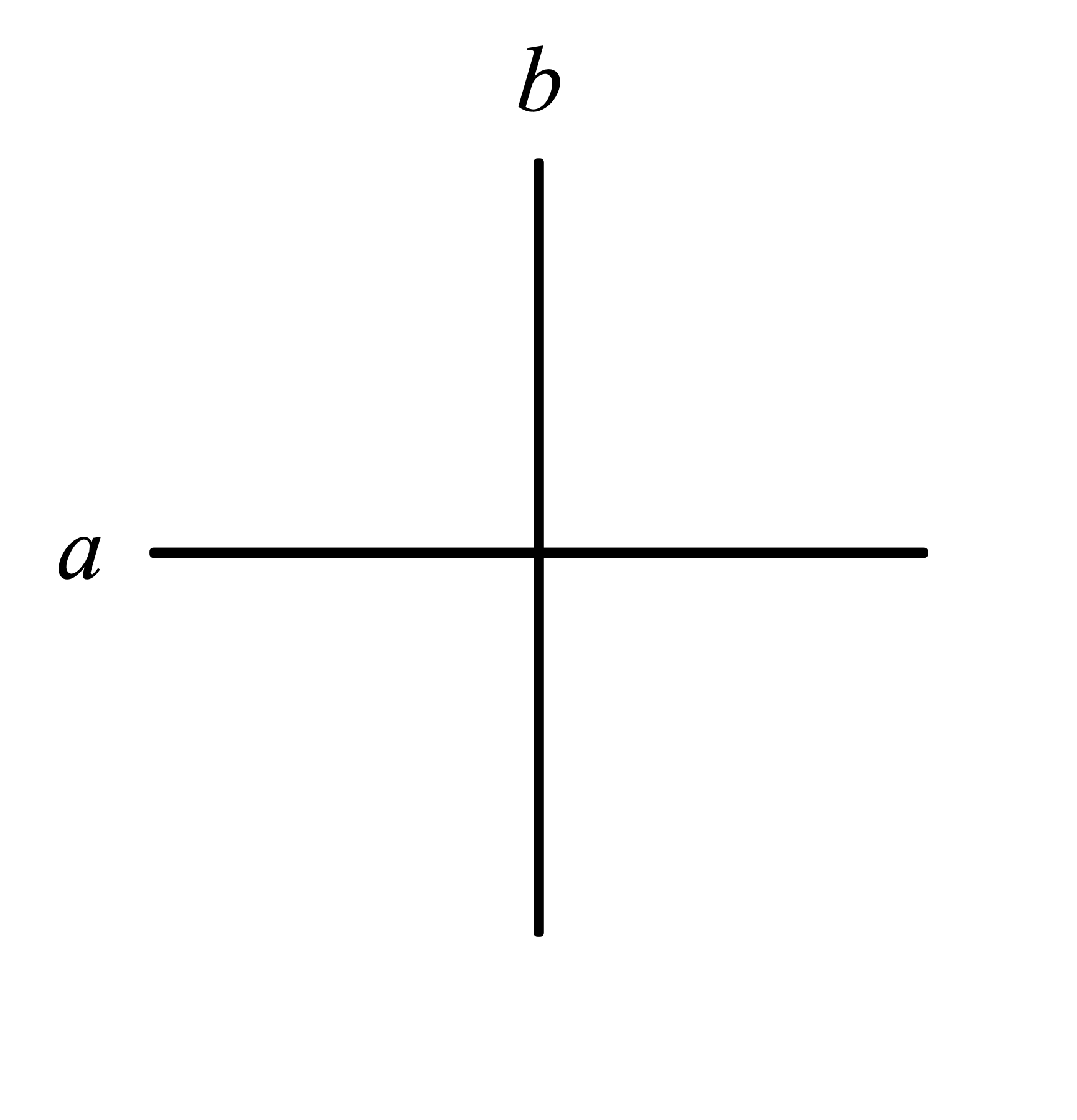} \\ \vspace{-1.5em} $xy = 0$}
    & a,b & (a,b \geq 2) & R & \ZZ/(a) \x \ZZ/(b) \\
    & \\
    \hline
\SetCell[r=6]{t,c,mode=text}
{\img[0.22]{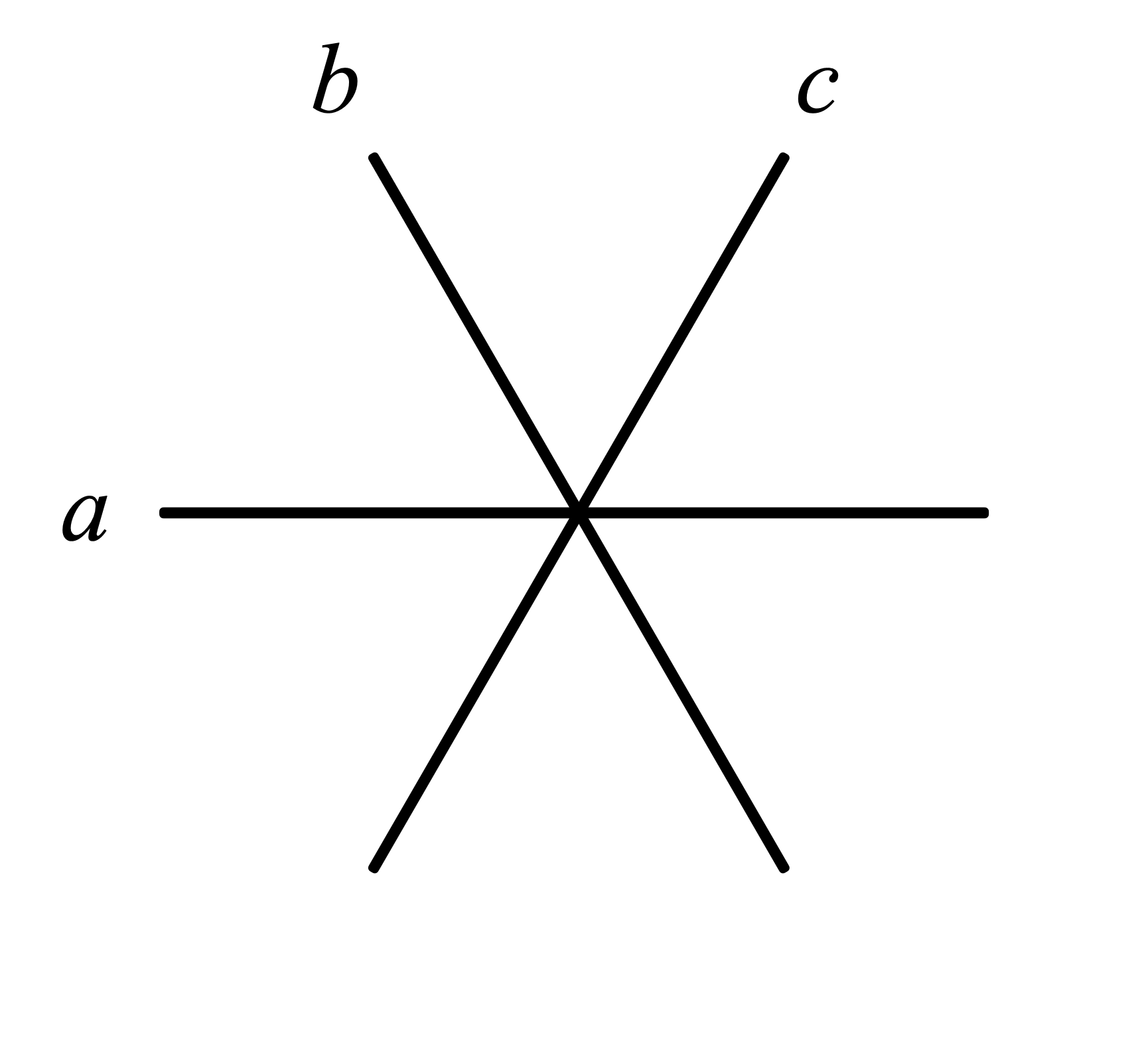} \\ \vspace{-1.2em} $y(y+x)(y-x) = 0$}
    & \underline{a,b,c}\!: \\
    & 2,2,c & (c \geq 2) & 2 & G(2c,2,2) \iso \ang{2,2,c}_c \\
    & 2,3,3 & & 7  & \ang{2,3,3}_6 \\
    & 2,3,4 & & 11 & \ang{2,3,4}_{12} \\
    & 2,3,5 & & 19 & \ang{2,3,5}_{30} \\
    & \\
    \hline
\SetCell[r=10]{t,c,mode=text}
{\img[0.22]{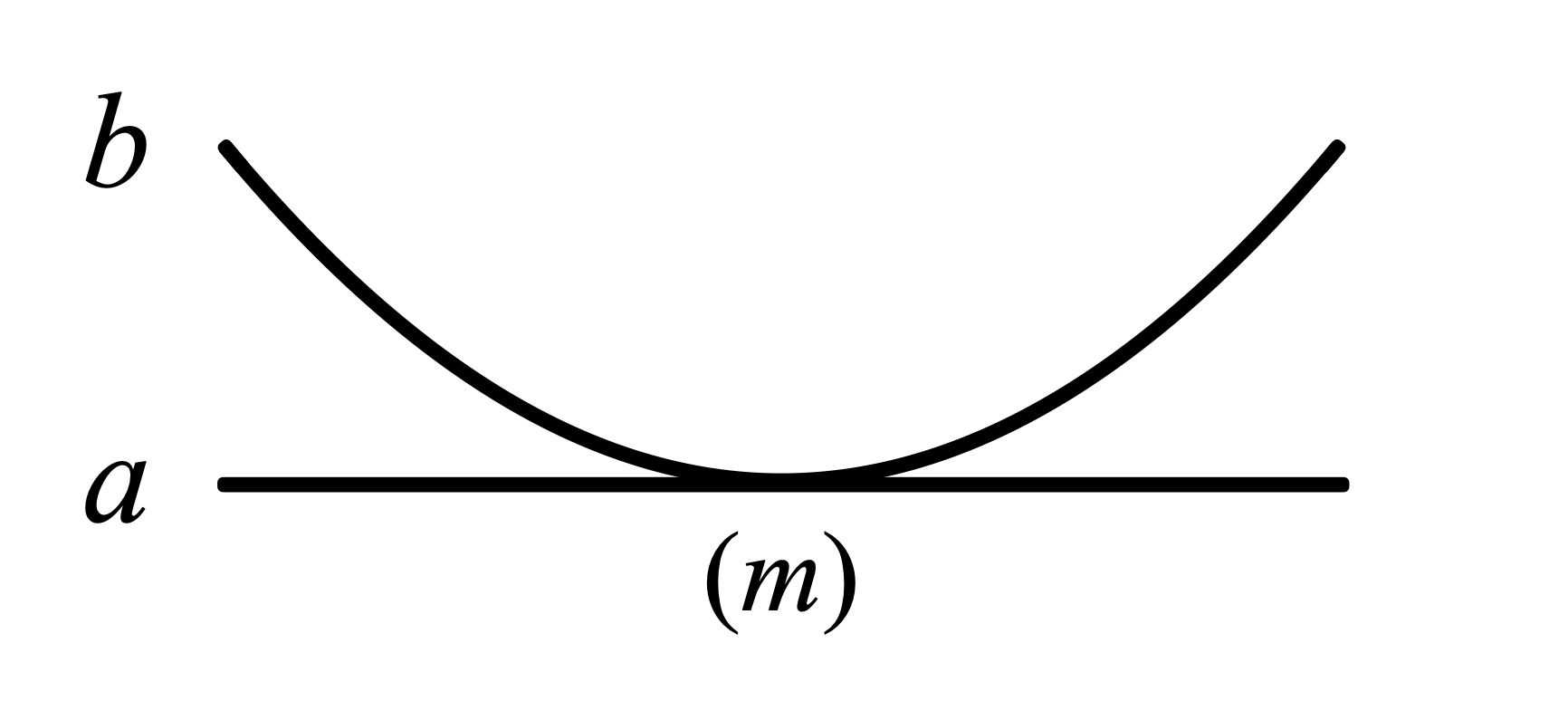} \\ \vspace{-0.3em} $y(y - x^m) = 0$}
    & \underline{a(m)b}\!: & & \\
    & 2(2)b & (b \geq 2) & 2 & G(b,1,2) \iso 2[4]b \\
    & 2(m)2 & (m \geq 2) & 2 & G(2m,2m,2) \iso 2[2m]2 \\
    & 2(3)3 & & 6  & 2[6]3 \\
    & 3(2)3 & & 5  & 3[4]3 \\
    & 2(4)3 & & 14 & 2[8]3 \\
    & 2(3)4 & & 9  & 2[6]4 \\
    & 3(2)4 & & 10 & 3[4]4 \\
    & 2(5)3 & & 21 & 2[10]3 \\
    & 2(3)5 & & 17 & 2[6]5 \\
    & 3(2)5 & & 18 & 3[4]5 \\
    \hline
\SetCell[r=4]{t,c,mode=text}
{\img[0.22]{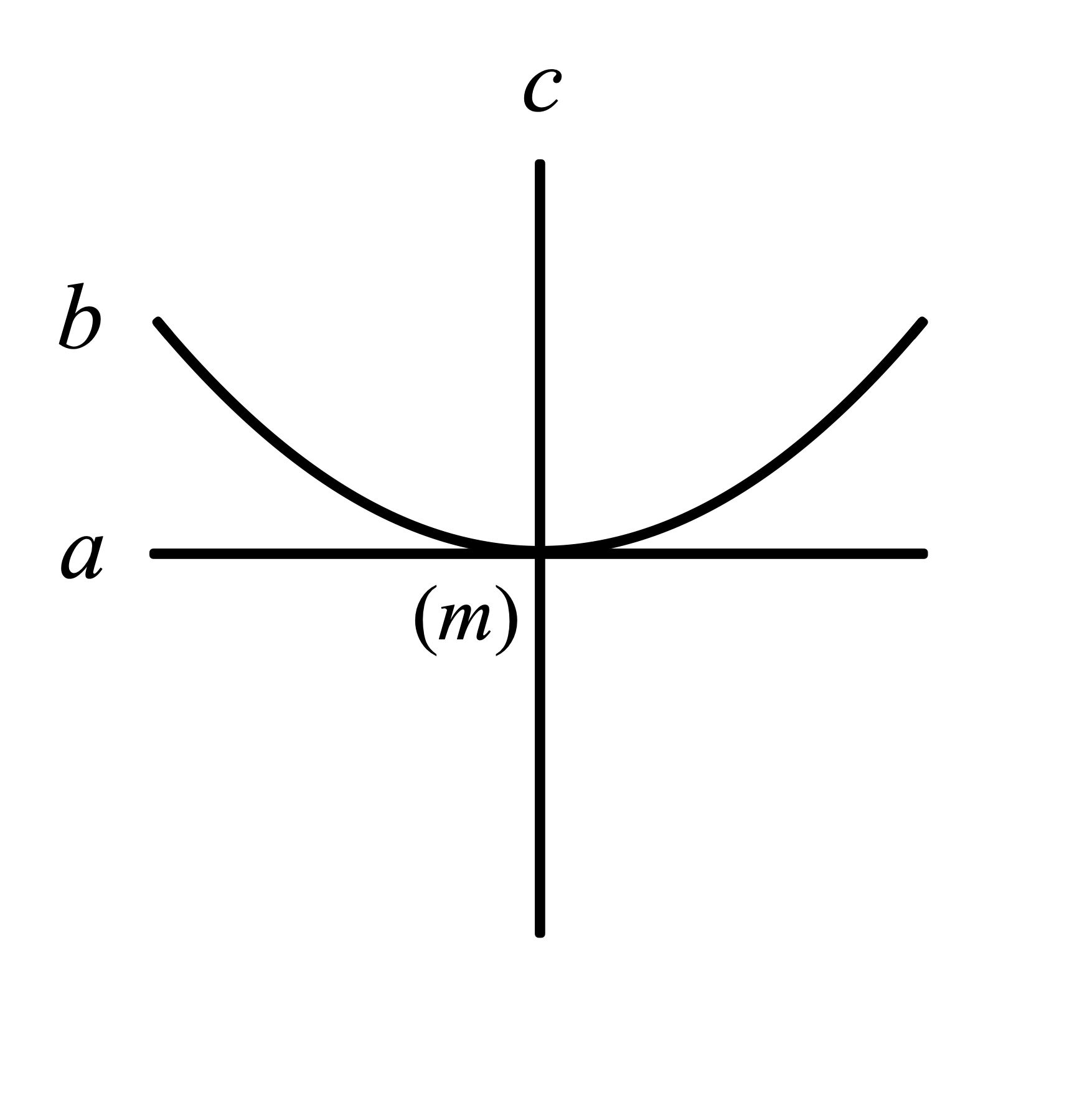} \\ \vspace{-1.5em} $xy(y-x^m) = 0$}
    & \underline{a(m)b,c}\!: & & \\
    & 2(m)2,c &(m,c \geq 2) & 2 & G(2mc, 2m, 2) \\
    & 2(2)3,2 & & 15 & \ang{2,3,4}_6 \\
    & \\
    \hline\pagebreak
\SetCell[r=7]{t,c,mode=text}
{\img[0.22]{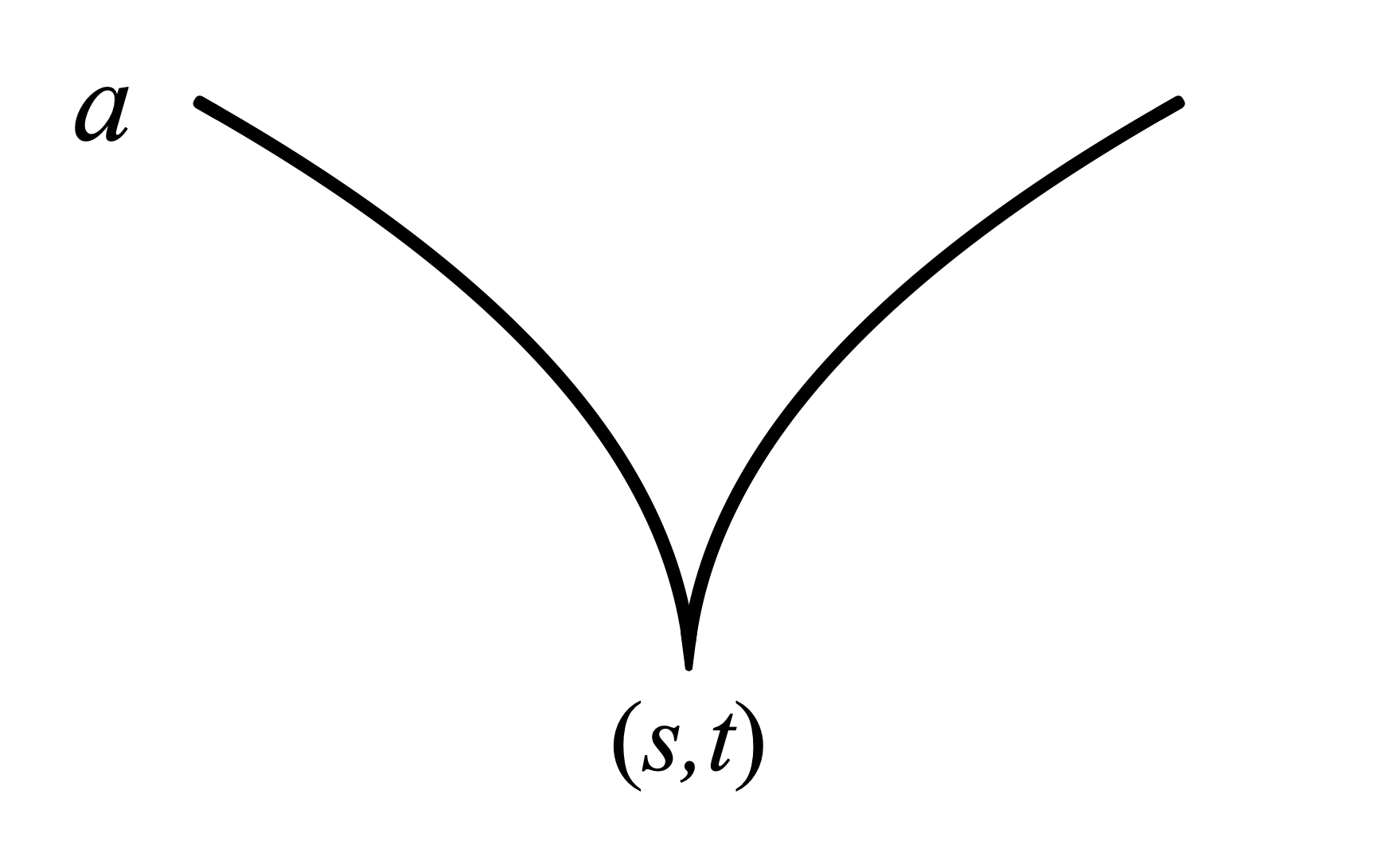} \\ \vspace{-0.3em} $y^t - x^s = 0$}
    & \underline{a_{(s,t)}}\!: & & \\
    & 2_{(2,t)} & (t \text{ odd} \geq 3) & 2 & G(t,t,2) \iso 2[t]2 \\
    & 3_{(2,3)} & & 4 & \ang{2,3,3}_1 \iso 3[3]3 \\
    & 2_{(3,4)} & & 12 & \ang{2,3,4}_1 \\
    & 4_{(2,3)} & & 8 & 4[3]4 \\
    & 2_{(3,5)} & & 22 & \ang{2,3,5}_2 \\
    & 3_{(2,5)} & & 20 & 3[5]3 \\
    & 5_{(2,3)} & & 16 & 5[3]5 \\
    \hline
\SetCell[r=3]{t,c,mode=text}
{\img[0.22]{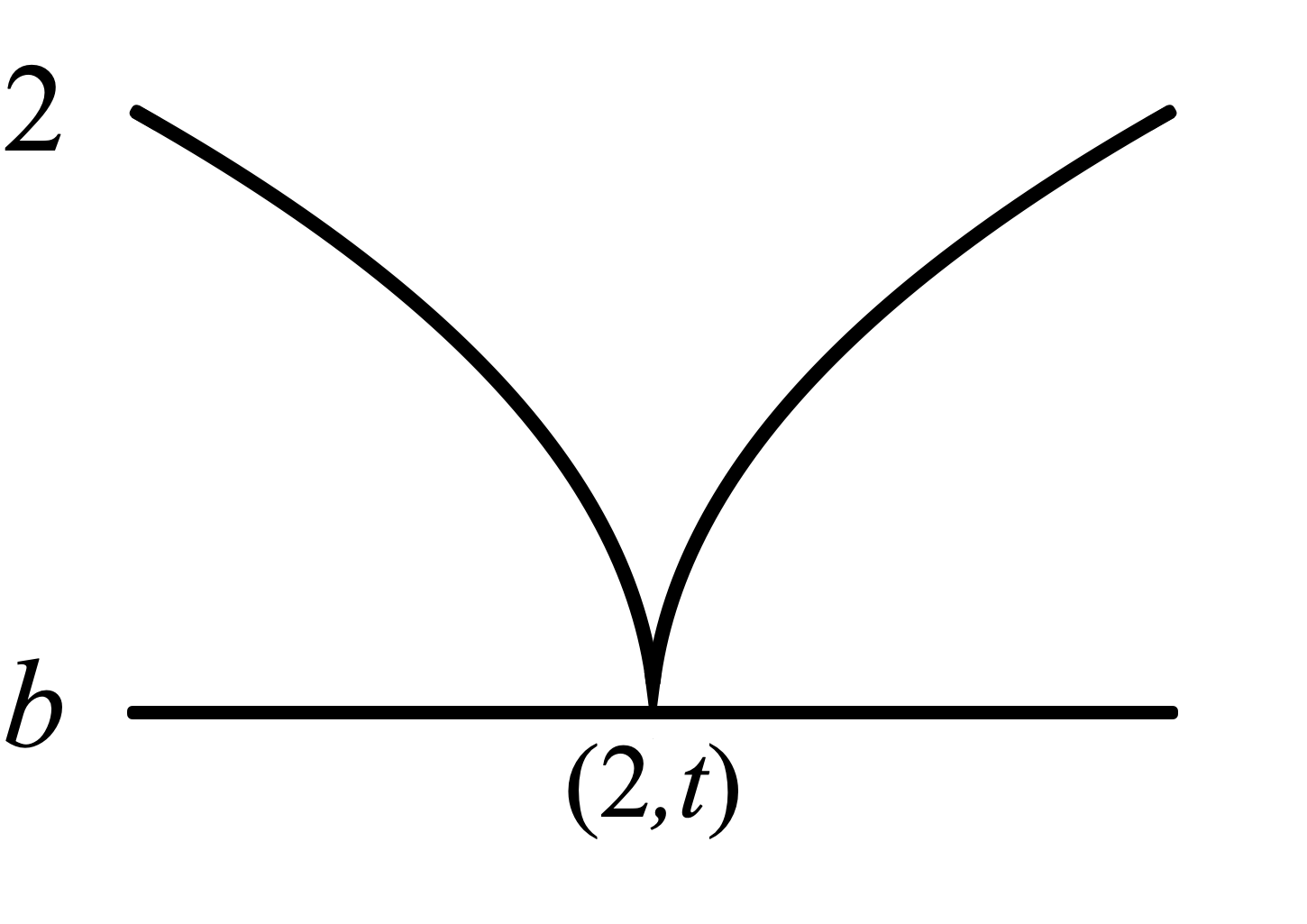} \\ \vspace{-0.2em} $y(y^t - x^2) = 0$}
    & 2_{(2,t)},b & \SetCell[r=2]{h,r,mode=text}{$(t \text{ odd} \geq 3,$ \\ $b \geq 2)$} & 2 & G(tb, t, 2) \\
    &&& \\
    & \\
    \hline
\SetCell[r=2]{t,c,mode=text}
{\img[0.22]{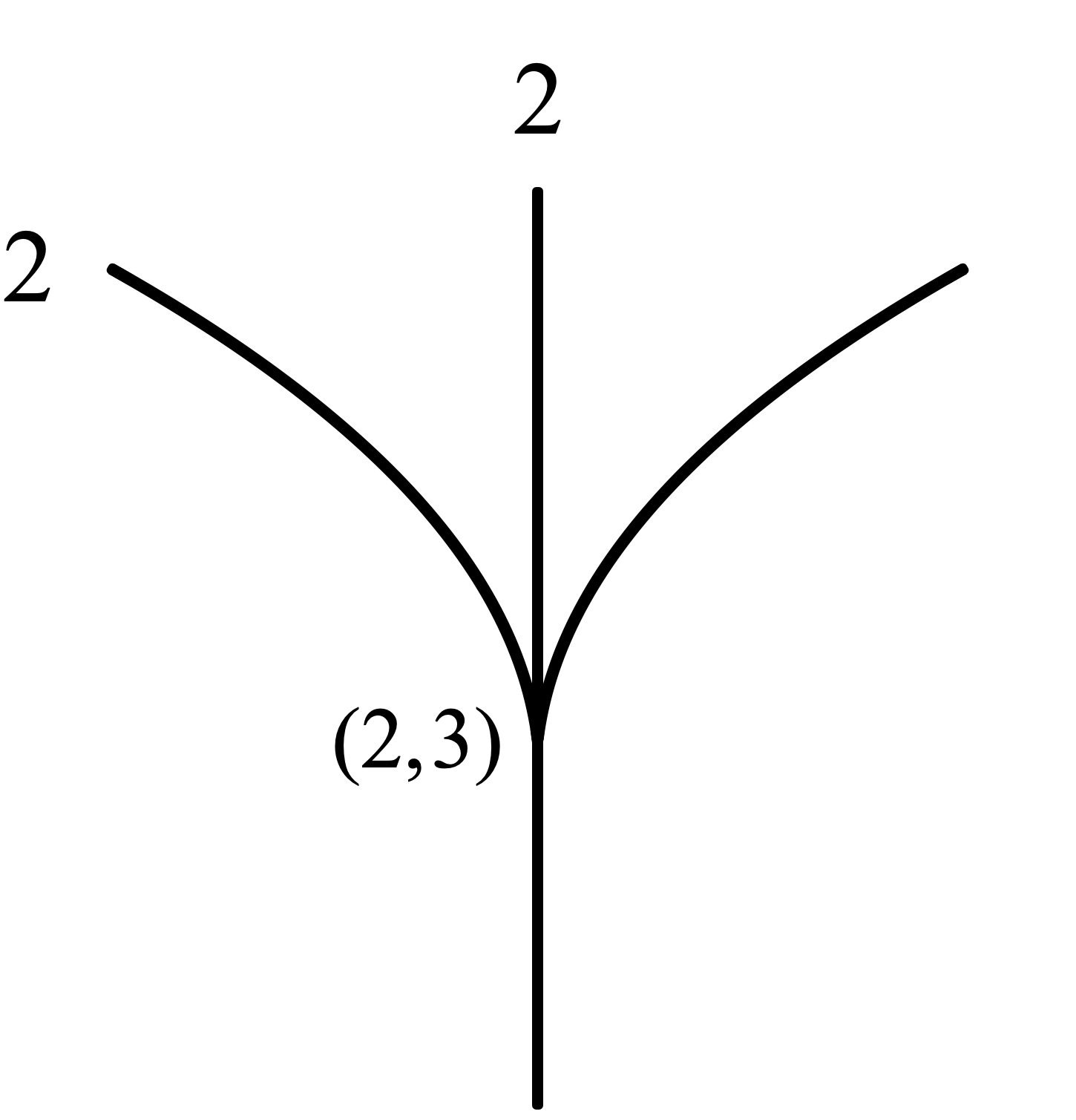} \\ $x(y^3 - x^2) = 0$}
    & 2_{(2,3)}(3)2 & & 13 & \ang{2,3,4}_2 \\
    & \\
\end{longtblr}
\vspace{-1ex} 
\emph{References for \cref{reflection-group-singularities}.}\enspace
The original classification of complex reflection groups is given in \cite{shephard-todd-1954-fur}; alternative expositions with emphasis on the dimension 2 classification include \cite[Chapters 2, 5, and 6]{lehrer-taylor-2009-urg} and \cite[Sections 2 and 3]{cohen-1976-fcr}. All three sources derive the order, the number of reflections, and a set of basic invariants. The branch structures may be deduced from a set of basic invariants, cf. \cite[\S 2]{bannai-1976-fgs}, \cite[\S 11.2]{yoshida-1987-fde}, or \cite[Table 1]{uludag-2005-gcp}. For group structure presentations, see the diagrams in \cite[Section 5]{broue-ea-1995-crg}. 

\emph{Notation in \cref{reflection-group-singularities}.}\enspace
Complex reflection groups are listed in the \textbf{S-T} column by their Shephard--Todd number (as assigned in \cite{shephard-todd-1954-fur}) with `R' standing for a reducible reflection group. Every rank 2 irreducible reflection group has one of the following group structures:
    \begin{itemize}
    \item
        The groups $G(m,p,2)$ belong to a large infinite family $G(m,p,n)$ of imprimitive reflection groups, all given the Shephard--Todd number 2. See e.g. \cite[Chapter 2]{lehrer-taylor-2009-urg} for a complete exposition. 
    \item
        $\ang{p,q,r}_s$ denotes the $s$-ary polyhedral group 
        \begin{gather*}
            \ang{A,B,C,Z: A^p = B^q = C^r = ABC = Z^s,\ [Z,A] = [Z,B] = [Z,C] = 1}.
        \end{gather*}
    \item
        $p[q]r$ denotes the Shephard group  
        \begin{gather*}
            \big\langle A,B : A^p = B^r = 1,\ \underbrace{ABA\cdots}_q = \underbrace{BAB\cdots}_q \big\rangle.
        \end{gather*}
    \end{itemize}

\newpage
\begin{notation}(Reflection group singularities)\label{singularity-notation}
Throughout this paper, reflection group singularities will be referred to by the symbols listed in \cref{reflection-group-singularities}, enclosed in square brackets, e.g. $[2,3,5]$. A symbol containing variables such as $[a,b,c]$ or $[2,2,c]$ refers to the family of reflection group singularities of that shape unless the variables are otherwise meaningful in context.

Singularity pairs $(T,C)$ will be denoted by making bold the symbols representing the branch curve $C$. For example, $[\mathbf{2},3,5]$ denotes the pair consisting of the singularity $[2,3,5]$ and the curve with branch index $2$. Note that the general theory in \cref{sec:computational-framework} allows for singularity pairs with locally reducible germs, such as $[\mathbf{2},\mathbf{2},3]$, although these do not arise in the main proof in \cref{sec:main-theorem-proof}.

A singularity configuration on a divisor or curve is denoted by a list of symbols, with superscripts denoting multiplicity. For example, $D_Q$ (see \cref{fig:line-arrangements}) has singularity configuration consisting of one triple point of type $[2,2,2]$, three triple points of type $[2,3,3]$, and three normal-crossings of type $[2,3]$, denoted
    \begin{gather*}
        [2,2,2]\,[2,3,3]^3\,[2,3]^3.
    \end{gather*}
\end{notation}

\section{Numerical constraints on ball quotient orbifolds}
\label{sec:computational-framework}

This section introduces ball quotients and presents several numerical constraints satisfied by ball quotient orbifolds $(X,D)$. The theory in this section is implicit in the literature; the purpose of this section is to re-work existing results into directly applicable forms.

\subsection{Ball quotients and ball quotient orbifolds}

The open unit ball $\BB^n$ in $\CC^n$ admits a complex-hyperbolic metric (the Bergman metric) such that
\begin{gather*}
    \Isom_\CC \BB^n = \Aut \BB^n \iso \PU_{1,n}(\CC)
\end{gather*}
(here $\Isom_\CC \BB^n$ denotes the holomorphic isometry group).

A \emph{ball quotient} is a finite-volume quotient $\BB^n / \Gamma$ where $\Gamma < \PU_{1,n}(\CC)$ is a discrete group of biholomorphic isometries. A ball quotient is called \emph{torsion-free} or \emph{torsion} according to whether its defining lattice $\Gamma$ is torsion-free or not.

In general, a ball quotient $X$ is a normal quasi-projective variety and not necessarily smooth. If $\Gamma$ is torsion-free then the quotient map $\BB^n \to X$ is the universal covering of $X$, so $X$ is a complete complex-hyperbolic manifold (in particular, smooth and aspherical).

When $\Gamma$ has torsion, the quotient map $\BB^n \to X$ is a branched cover whose branch divisor $D$ endows $X$ with the structure of a complex-hyperbolic orbifold $(X,D)$. An orbifold divisor $D$ arising in this manner is called a \emph{ball quotient divisor} on $X$; the pair $(X,D)$ is called a \emph{ball quotient orbifold}.

\begin{remark}[Any ball quotient orbifold is finitely uniformized by a torsion-free ball quotient]
\label{ball-quotients:finite-uniformization}
    For any ball quotient orbifold $(X,D) = \BB^n / \Gamma$, there exists a torsion-free ball quotient $Y = \BB^n / \Lambda$ such that the quotient map $\BB^n \to X$ factors as a composition of normal covers
    \begin{gather*}
        \BB^n \xrightarrow[\text{unbranched}]{\Lambda} Y \xrightarrow[\text{finite branched}]{\Gamma/\Lambda} X.
    \end{gather*}
    This follows from Selberg's Lemma: Any lattice $\Gamma < \PU_{1,n}(\CC)$ admits a normal finite-index torsion-free subgroup $\Lambda$. The deck group of the map $Y \to X$ is the finite group $G \defeq \Gamma / \Lambda < \Aut Y$ and $Y / G = (X,D)$ as orbifolds.
\end{remark}

\newpage 
The following is well-known and included here for completeness.

\begin{proposition}
\label{ball-quotients:torsion-free-ample}
    A compact torsion-free ball quotient is a projective variety with ample canonical class.
\end{proposition}
\begin{proof}
The Bergman metric on $\BB^n$ is K\"ahler and has constant negative holomorphic curvature (a standard exposition can be found in e.g. \cite[Section 3.1]{goldman-1999-chg}). 

If $X$ is a torsion-free ball quotient then the quotient map $\BB^n \to X$ is an unbranched covering map so $X$ is smooth and inherits a K\"ahler metric with constant negative holomorphic curvature $\kappa < 0$. Let $\omega$ denote the K\"ahler form on $X$. The Ricci form on $X$ satisfies
\begin{gather*}
    \Ric_\omega = \frac 12 (n+1) \kappa \omega
\end{gather*}
by e.g. \cite[Remark on p. 168]{kobayashi-nomizu-1996-fdg2}. On the other hand, a property of the Ricci form is that
\begin{gather*}
    c_1(X) = \l[ \frac{1}{2\pi} \Ric_\omega \r]
\end{gather*}
so then $c_1(K_X) = -c_1(X)$ is a positive multiple of the K\"ahler form $\omega$. When $X$ is compact, Kodaira's embedding theorem implies that $K_X$ is ample. In particular, $X$ is projective of general type.
\end{proof}

\begin{corollary}
\label{ball-quotients:orbifold-ample}
    If $(X,D)$ is a compact ball quotient orbifold, then $K_X + D$ is ample as a $\QQ$-divisor.
\end{corollary}

\begin{proof}
    Let $\pi: Y \to X$ be a finite uniformization of $(X,D)$ by a torsion-free ball quotient $Y$. Then $K_Y$ is ample and $K_Y = \pi^*(K_X + D)$ by \cref{lift-canonical}, so $K_X + D$ is also ample because pullbacks along finite surjective maps do not change ampleness \cite[Proposition 1.2.13, Corollary 1.2.28]{lazarsfeld-2004-pag}.
\end{proof}

\subsection{Hirzebruch proportionality and Yau's theorem for orbifolds}

The objective of this subsection is to introduce the main computational tools for this paper.

For the purpose of eliminating ball quotient structures: \cref{orbifold-prop-definitions} and \cref{rem:prop-derived} define several numerical invariants associated to any smooth compact orbifold that vanish for ball quotient orbifolds (stated precisely in \cref{constraints}). The vanishing results are obtained by adapting Hirzebruch proportionality to the orbifold setting.

For the purpose of verifying that a certain orbifold is a ball quotient: \cref{ball-quotients:orbifold} gives a complete characterization of ball quotient orbifolds in terms of numerical invariants. This is essentially Hirzebruch proportionality for orbifolds, combined with an orbifold version of Yau's theorem, due to Kobayashi--Nakamura--Sakai.
\medskip

The following results on compact torsion-free ball quotients are due to Hirzebruch and Enoki respectively:

\begin{theorem}[Hirzebruch proportionality {\cite[Satz 3]{hirzebruch-1958-afu}}]
\label{hirzebruch-proportionality}
    If $Y$ is a compact torsion-free ball quotient surface then
    \vspace{-0.5ex} 
    \begin{gather*}
        c_1^2(Y) = 3c_2(Y).
    \end{gather*}
\end{theorem}
\smallskip 

\begin{theorem}[A version of Hirzebruch proportionality for curves in surfaces {\cite[B.1.I, attributed to Enoki]{barthel-ea-1987-gua}}]
\label{relative-proportionality}
    If $Y$ is a compact torsion-free ball quotient surface then any smooth totally geodesic curve $C$ on $Y$ satisfies
    \begin{gather*}
        e(C) = 2 C^2.
    \end{gather*}
\end{theorem}

\begin{definition}
For a smooth compact surface $Y$ and a smooth irreducible curve $C$, Hirzebruch defined
\begin{gather*}
    \Prop(Y) \defeq 3c_2(Y) - c_1^2(Y) \\
    \prop_Y(C) \defeq 2C^2 - e(C)
\end{gather*}
so that \cref{hirzebruch-proportionality,relative-proportionality} may be viewed as vanishing results.
\end{definition}
\medskip

By \cref{ball-quotients:finite-uniformization}, any compact ball quotient orbifold $(X,D)$ is of the form $Y / G$ for a compact torsion-free ball quotient $Y$ and a finite group $G < \Aut Y$. The aim is to use the vanishing results on $Y$ to obtain vanishing results purely in terms of the orbifold $(X,D)$. To this end, the following proposition gives formulas for the Chern numbers and intersection numbers on $Y$ in terms of the invariants of $X$ and the orbifold divisor $D$, in the case that $X$ is smooth.

{ \setlength{\mathindent}{0.96cm}
  
\begin{proposition}
\label{computation}
Let $Y$ be a smooth compact surface and suppose that $G < \Aut Y$ is a finite group such that $X \defeq Y / G$ is smooth. Let $\pi: Y \to X$ denote the quotient map and let $D = \sum_i \big( 1 - \frac{1}{b_i} \big) D_i$ denote the branch divisor of $\pi$.
\begin{parts}
\item
\label{computation:chern-numbers}
    The Chern numbers of $Y$ are:
    \begin{gather*}
        c_1^2(Y) = |G| (K_X + D)^2 \\
        c_2(Y) = |G| \l[  c_2(X) - \sum_i \l(1 - \frac{1}{b_i} \r) \l( e(D_i) - |D_\sing \cap D_i| \r) - \sum_{x \in D_\sing} \l( 1 - \frac{1}{|\Gamma_x|} \r) \r]
    \end{gather*}
    where $\Gamma_x$ denotes the orbifold local group at $x$, with respect to the orbifold $(X,D)$. In particular,
    \begin{gather*}
        \Prop(Y) = |G| \l[ \Prop(X) + \sum_i \l(1 - \frac{1}{b_i}\r) \l( K_X \cdot D_i + \l( 2 + \frac{1}{b_i} \r) D_i^2 \r) + \sum_{x \in D_\sing} \lambda_{x,D} \r]
    \end{gather*}
    where $\lambda_{x,D}$ is a coefficient defined in \eqref{def:lambda} with the property that if every irreducible component of $D$ is locally irreducible, then $\lambda_{x,D}$ depends only on the isomorphism type of $x$ as a weighted singularity.

\item
\label{computation:intersection-numbers}
    For any irreducible branch curve $D_1$ in $D$, and any irreducible component $\tilde D_1$ of $\pi^{-1}(D_1)$, the intersection numbers of $D_1$ and $\tilde D_1$ are related as follows:
\begin{gather*}
\begin{split}
    K_Y \cdot \tilde D_1
    &= \frac{|\stab(\tilde D_1)|}{b_1} \xb3[
        K_X \cdot D_1 + \l( 1 - \frac{1}{b_1} \r) D_1^2
        + \sum_{x \in D_\sing \cap D_1} \sum_{\substack{j \neq 1 \\ x \in D_j}} \l( 1 - \frac{1}{b_j} \r) \ihat_x(D_j, D_1)
        \xb3] \\
\end{split} \\
    \tilde D_1^2
    = \frac{|\stab(\tilde D_1)|}{b_1} \l[
        \frac{1}{b_1} D_1^2 - \sum_{x \in D_\sing \cap D_1} \l( \frac{b_1\rho_{x,D_1}(\rho_{x,D_1} - 1)}{|\Gamma_x|} \r)
    \r]
\end{gather*}
\vspace{-2.3ex} 

where
\begin{align*}
    &\stab(C)
    \defeq \text{setwise stabilizer of a curve $C$ under the action of $G$} \\
    &\ihat_x(C,C')
    \defeq \text{algebraic intersection number of curves $C$ and $C'$ at a point $x$} \\
    &\rho_{x,C}
    \defeq \text{number of mirrors in the preimage of a curve $C$, in the linear local model at $x$}
\end{align*}
In particular,
\vspace{-0.6ex} 
\begin{gather*}
    \prop_Y(\tilde D_1) = \frac{|\stab(\tilde D_1)|}{b_1} \xb3[ K_X \cdot D_1 + \l( 1 + \frac{2}{b_1} \r) D_1^2 + \sum_{x \in D_\sing \cap D_1} \alpha_{x,D_1} \xb3]
\end{gather*}
where $\alpha_{x,D_1}$ is a coefficient defined in \eqref{def:alpha} that depends only on the isomorphism type of $(x,D_1)$ as a singularity pair (without any local irreducibility assumptions on the components of $D$).
\end{parts}
\end{proposition}
}

\cref{computation} (whose proof is postponed to \cref{sec:computation-proof}) motivates the following orbifold definitions:

\begin{definition}[Orbifold proportionality]
\label{orbifold-prop-definitions}
For a smooth compact orbifold $(X,D)$, define
\begin{align}
    \Prop(X,D) &\defeq \Prop(X) + \sum_i \l(1 - \frac{1}{b_i}\r) \l( K_X \cdot D_i + \l( 2 + \frac{1}{b_i} \r) D_i^2 \r) + \sum_{x \in D_\sing} \lambda_{x,D} \label{def:prop-orbifold}
\end{align}
with $\lambda_{x,D}$ as defined in \eqref{def:lambda} (computed for relevant cases in \cref{table:singularity-coefficients}). For an irreducible component $D_i$ of the orbifold divisor $D$, define the invariant
\begin{gather}
    \prop_{(X,D)}(D_i) \defeq K_X \cdot D_i + \l( 1 + \frac{2}{b_i} \r) D_i^2 + \sum_{x \in D_\sing \cap D_i} \alpha_{x,D_i} \label{def:prop-curve}
\end{gather}
with $\alpha_{x,D_i}$ as defined in \eqref{def:alpha} (computed for relevant cases in \cref{table:singularity-pair-coefficients}).
\end{definition}
\smallskip

The desired vanishing results on $(X,D)$ can now be stated cleanly: Suppose that $(X,D)$ is a smooth compact ball quotient orbifold. Let $Y$ be a torsion-free ball quotient finitely uniformizing $(X,D)$. Then \cref{hirzebruch-proportionality} applies to $Y$. With the notation of \cref{computation}, 
\begin{gather*}
    \Prop(X,D) = \frac{1}{|G|} \Prop(Y) = 0.
\end{gather*}
Also, for any irreducible component $D_1$ of $D$, any irreducible component $\tilde D_1$ of $\pi^{-1}(D_1)$ is smooth and totally geodesic because any irreducible ramification curve is a connected component of the fixed set of some finite-order biholomorphic isometry. Then by \cref{relative-proportionality},
\begin{gather*}
    \prop_{(X,D)}(D_1) = \frac{b_1}{|\stab(\tilde D_1)|} \prop_Y(\tilde D_1) = 0.
\end{gather*}
In summary:
\vspace{-0.5ex} 

\begin{proposition}
\label{constraints}
    If $(X,D)$ is a smooth compact ball quotient orbifold then
    \begin{parts}[parsep=2pt]
    \item \label{constraints:orbifold}
        $\Prop(X,D) = 0$; and
    \item \label{constraints:curve}
        $\prop_{(X,D)}(D_i) = 0$ for all irreducible components $D_i$ of $D$.
    \end{parts}
\end{proposition}

\begin{remark}
\label{rem:prop-derived}
    It is also useful to define the orbifold invariant
    \begin{align}
    \label{def:prop-derived}
        \sP(X,D) &\defeq \Prop(X,D) - \frac 12 \sum_i \l( 1 - \frac{1}{b_i} \r) \prop_{(X,D)}(D_i) \nonumber\\
                &= \Prop(X) + \frac 12 \sum_i \l( 1 - \frac{1}{b_i} \r) \l( K_X \cdot D_i + 3D_i^2 \r) + \sum_{x \in D_\sing} \theta_{x,D}
    \end{align}
    \vspace{-3.5ex}
    
    where
    \begin{gather}
    \label{def:theta}
        \theta_{x,D} \defeq \lambda_{x,D} - \frac 12 \sum_{i: x \in D_i} \l(1 - \frac{1}{b_i}\r) \alpha_{x,D_i},
    \end{gather}
    because $\theta_{x,D} = 0$ for any singularity $x \in D_\sing$ if the components of $D$ are smooth and pairwise normal-crossing (see \cref{table:singularity-coefficients}). In general, if the irreducible components of $D$ are locally irreducible, then $\theta_{x,D}$ depends only on the isomorphism type of $x$ as a weighted singularity. Note that if $(X,D)$ is a ball quotient orbifold then $\sP(X,D) = 0$ by \cref{constraints}.
\end{remark}

\paragraph{Yau's theorem and Kobayashi--Nakamura--Sakai's orbifold analogue.}

Remarkably, a converse of Hirzebruch proportionality holds: Yau showed that if a smooth compact surface $Y$ with ample canonical class satisfies $c_1^2(Y) = 3c_2(Y)$, then $Y$ is necessarily a torsion-free ball quotient \cite[Theorem 4]{yau-1977-ccn}. Together with \cref{hirzebruch-proportionality} (Hirzebruch proportionality) and \cref{ball-quotients:torsion-free-ample}, this yields a complete characterization of
\newpage 
compact torsion-free ball quotient surfaces:

\begin{theorem}[Yau {\cite{yau-1977-ccn}}, Hirzebruch {\cite{hirzebruch-1958-afu}}]
\label{ball-quotients:torsion-free}
    A smooth compact surface $Y$ is a torsion-free ball quotient if and only if $K_Y$ is ample and
    \begin{gather*}
        c_1^2(Y) = 3c_2(Y).
    \end{gather*}
\end{theorem}

There is an analogous characterization for ball quotient orbifold surfaces. The \emph{orbifold Chern numbers} of a smooth compact orbifold surface are given by
    \begin{align*}
        c_1^2(X,D)
        &\defeq (K_X + D)^2 \\
        c_2(X,D)
        &\defeq c_2(X) - \sum_i \l(1 - \frac{1}{b_i} \r) \l( e(D_i) - |D_\sing \cap D_i| \r) - \sum_{x \in D_\sing} \l( 1 - \frac{1}{|\Gamma_x|} \r)
    \end{align*}
(cf. \cref{computation:chern-numbers}). Then \cref{constraints:orbifold} can be restated as follows: If $(X,D)$ is a smooth compact ball quotient orbifold then $c_1^2(X,D) = 3c_2(X,D)$. Thus \cref{constraints:orbifold} should be viewed as an orbifold analogue of Hirzebruch proportionality. 

Kobayashi--Nakamura--Sakai proved a direct analogue of Yau's theorem in the orbifold case \cite{kobayashi-1990-ucs,kobayashi-ea-1989-ncb} (see also the reference notes below \cref{ball-quotients:orbifold}). The following characterization of ball quotient orbifolds is Kobayashi--Nakamura--Sakai's result, combined with Hirzebruch proportionality for orbifolds (as stated in \cref{constraints:orbifold}) and \cref{ball-quotients:orbifold-ample}.

\begin{theorem}[Kobayashi--Nakamura--Sakai {\cite{kobayashi-1990-ucs,kobayashi-ea-1989-ncb}}, Hirzebruch {\cite{hirzebruch-1958-afu}}]
\label{ball-quotients:orbifold}
    A smooth compact orbifold surface $(X,D)$ is a ball quotient orbifold if and only if $K_X+D$ is ample and
    \begin{gather*}
        c_1^2(X,D) = 3c_2(X,D).
    \end{gather*}
\end{theorem}

\begin{proof}[Reference notes.] 
    \cref{ball-quotients:orbifold} is obtained as the equality case of the inequality given in \cite[Section 3.0, Theorems 1 and 2]{kobayashi-1990-ucs}, in the special case when (in the notation of \cite{kobayashi-1990-ucs}) $X$ is a smooth surface and $b_i \neq \infty$ for all weights $b_i$ of $D$. The following remarks clarify the difference between the terminology used in the theorem statement in \cite{kobayashi-1990-ucs} and that of \cref{ball-quotients:orbifold}. Quotation marks indicate vocabulary from \cite{kobayashi-1990-ucs,kobayashi-ea-1989-ncb}.
    \begin{itemize}
    \item
    A ``normal surface pair'' is a pair $(X,D)$ consisting of a normal surface $X$ with a $\QQ$-divisor $D = \sum_i \big(1 - \frac{1}{b_i}\big) D_i$ where $b_i \in \ZZ_{\geq 2} \cup \{\infty\}$.
    \item
    When $X$ is a smooth compact surface, a ``log-terminal singularity'' is equivalent to a reflection group singularity by case (1)** of the classification given in \cite[Theorem 3.1]{kobayashi-1990-ucs}, attributed to Nakamura. A ``log-terminal singularity'' is ``log-canonical''.
    \item
    $K_X + D$ is ample if and only if $\overline{\mathrm{kod}}(X,D) = 2$ (an orbifold analogue of Kodaira dimension, defined in \cite{kobayashi-ea-1989-ncb}) and $(X,D)$ is equal to its ``log-canonical model'' $(X'',D'')$ \cite{kobayashi-ea-1989-ncb}. In particular, $(X_0'',D_0'') = (X'',D'') = (X,D)$ in the case that $X$ is a smooth surface and $b_i \neq \infty$ for all $i$. \qedhere
\end{itemize}
\end{proof}

\subsection{Proof of \cref{computation}}
\label{sec:computation-proof}

Let $Y$ be a smooth compact surface and let $G < \Aut Y$ be a finite group such that $X \defeq Y / G$ is smooth. Let $\pi: Y \to X$ denote the quotient map and let $D = \sum_i \big( 1 - \frac{1}{b_i} \big) D_i$ be the branch divisor of $\pi$.

\begin{proof}[Proof of \cref{computation:chern-numbers}]
By definition, $c_1(Y) = c_1(TY) = - c_1(K_Y)$ so
\begin{gather*}
    c_1^2(Y) \defeq \ang{c_1(Y) \smile c_1(Y), [Y]} = (-K_Y)^2 = K_Y^2.
\end{gather*}
Then by \cref{lift-canonical},
\begin{gather*}
    c_1^2(Y) = K_Y^2 = \pi^*(K_X + D)^2 = |G| (K_X + D)^2.
\end{gather*}

\newpage 
    The second Chern number coincides with the topological Euler characteristic for a smooth compact surface.
    This paper uses $e(-)$ to mean the topological Euler characteristic of a space, in deference to sources such as \cite{barthel-ea-1987-gua,kobayashi-1990-ucs}. 
    Fix the following notation:
    \begin{alignat*}{2}
        &R &&\defeq \text{ramification locus of $\pi$}, \\
        &\{R_j\}_j &&\defeq \text{irreducible curve components of the ramification locus}, \\
        &R_{\sing} &&\defeq \text{set of singularities of the ramification locus}, \\
        &B &&\defeq \text{branch locus of $\pi$} = \text{support of $D$} = \pi(R).
    \end{alignat*}
    Note that $\pi(R_\sing) = D_\sing$.

    For a complex variety $X$ and a subvariety $Z$, the topological Euler characteristic has the property that
    \begin{gather} \label{inclusion-exclusion}
        e(X) = e(X - Z) + e(Z)
    \end{gather}
    (see e.g. \cite[Notes 4.13 on p.141--142]{fulton-2016-itv}). Then
    \begin{align*}
        e(Y) &= e(Y - R) + e(R) = e(Y - R) + \sum_j e(R_j - R_{\sing}) + |R_{\sing}|, \\
        e(X) &= e(X - B) + e(B) = e(X - B) + \sum_i \l( e(D_i) - |D_\sing \cap D_i| \r) + |D_\sing|.
    \end{align*}
    Since $\pi: Y - R \to X - \pi(R)$ is a degree $|G|$ unbranched cover,
    \begin{align*}
        e(Y-R) = |G| e(X - B) = |G| \bigg( e(X) - \sum_i \l( e(D_i) - |D_\sing \cap D_i| \r) - |D_\sing| \bigg).
    \end{align*}
    \vspace{-3ex}

    For a given branch curve $D_i$, the restriction
    \begin{gather*}
        \pi : \l( \pi^{-1}(D_i) - R_{\sing} \r) \to \l( D_i - D_\sing \r)
    \end{gather*}
    is an unbranched covering map of degree $\frac{|G|}{b_i}$ where $b_i$ is the branch index of $D_i$. Thus
    \begin{gather*}
        e \l( \pi^{-1}(D_i) - R_{\sing} \r)
        = \frac{|G|}{b_i} e(D_i - D_\sing)
        = \frac{|G|}{b_i} \l( e(D_i) - |D_\sing \cap D_i| \r).
    \end{gather*}
    Each irreducible ramification curve $R_j$ is contained in exactly one preimage $\pi^{-1}(D_i)$, so re-index:
    \begin{gather*}
        \sum_j e(R_j - R_{\sing})
        = \sum_i e(\pi^{-1}(D_i) - R_{\sing})
        = \sum_i \frac{|G|}{b_i} \l( e(D_i) - |D_\sing \cap D_i| \r)
    \end{gather*}
    Finally, for any point $x \in D_\sing$,
    \begin{gather*}
        |\pi^{-1}(x)| = \frac{|G|}{|\Gamma_x|}
    \end{gather*}
    by orbit-stabilizer. Then
    \begin{gather*}
        |R_{\sing}| = \sum_{x \in D_\sing} |\pi^{-1}(x)| = \sum_{x \in D_\sing} \frac{|G|}{|\Gamma_x|}.
    \end{gather*}
    Altogether,
    \begin{align*}
        e(Y)
        &= e(Y - R) + \sum_j e(R_j - R_{\sing}) + |R_{\sing}| \\[-1ex]
        &= |G| \Bigg[ e(X) - \sum_i \l(1 - \frac{1}{b_i} \r) \l( e(D_i) - |D_\sing \cap D_i| \r)
          - \sum_{x \in D_\sing} \l(1 - \frac{1}{|\Gamma_x|} \r) \Bigg].
    \end{align*}
\medskip

\newpage 
Now expand and rearrange the formula for $\Prop(Y)$ to isolate the contributions of curves and singularities:
{ \setlength{\mathindent}{0.96cm}
\begin{align*}
\frac{1}{|G|} \Prop(Y) 
&= \frac{1}{|G|} \l( 3c_2(Y) - c_1^2(Y) \r) \\
&= 3 \Bigg[ c_2(X) - \sum_i \beta_i \l( e(D_i) - |D_\sing \cap D_i| \r) - \sum_{x \in D_\sing} \l( 1 - \frac{1}{|\Gamma_x|} \r) \Bigg]
    \qquad\beta_i \defeq 1 - \frac{1}{b_i}  \\
    &\qquad\quad - \Bigg[ K_X^2 + 2 \sum_i \beta_i K_X \cdot D_i + \sum_{i,j} \beta_i \beta_j D_i \cdot D_j \Bigg] \\
&= \Prop(X) + \sum_i \Big[ - 3 \beta_i e(D_i) - 2 \beta_i K_X \cdot D_i - \beta_i^2 D_i^2 \Big]
      - \sum_{x \in D_\sing} 3\l( 1 - \frac{1}{|\Gamma_x|} \r) \\
    &\qquad\quad
      + \sum_i 3 \beta_i |D_\sing \cap D_i|
      - \sum_{i \neq j} \beta_i \beta_j D_i \cdot D_j \\
&\overset{\text{(a)}}{=}
    \Prop(X) + \sum_i \Big[ \beta_i K_X \cdot D_i + \beta_i(3 - \beta_i) D_i^2 \Big] 
      - \sum_{x \in D_\sing} 3\l( 1 - \frac{1}{|\Gamma_x|} \r) \\
    &\qquad\quad
      + \sum_i \Bigg[ 3 \beta_i |D_\sing \cap D_i| - 3\beta_i \sum_{x \in (D_i)_\sing} \mu_{x,D_i} \Bigg]
      - \sum_{i \neq j} \beta_i \beta_j D_i \cdot D_j \\
&\overset{\text{(b)}}{=}
    \Prop(X) + \sum_i \beta_i \Big( K_X \cdot D_i + (3 - \beta_i) D_i^2 \Big)
      - \sum_{x \in D_\sing} 3\l( 1 - \frac{1}{|\Gamma_x|} \r) \\
    &\qquad\quad
      + \sum_{x \in D_\sing} \sum_{i: x \in D_i} 3 \beta_i \l( 1 - \mu_{x,D_i} \r)
      - \sum_{x \in D_\sing} \sum_{i \neq j : x \in D_i \cap D_j} \beta_i \beta_j \ihat_x(D_i,D_j) \\
    &= \Prop(X) + \sum_i \beta_i \l( K_X \cdot D_i + (3 - \beta_i) D_i^2 \r) + \sum_{x \in D_\sing} \lambda_{x,D}
\end{align*}
}
where
\begin{align}
    &\mu_{x,C}
    \defeq \text{the Milnor number of $x$ as a singular point of $C$}, \nonumber \\
    &\lambda_{x,D}
    \defeq - 3 \l(1 - \frac{1}{|\Gamma_x|} \r) + \sum_{i: x \in D_i} 3 \beta_i \l( 1 - \mu_{x,D_i} \r) - \sum_{i \neq j: x \in D_i \cap D_j} \beta_i \beta_j \ihat_x(D_i, D_j).
    \label{def:lambda}
\end{align}
Here, $\ihat_x(C,C')$ denotes the algebraic intersection number of curves $C$ and $C'$ at a point $x$.

The labelled steps deserve elaboration:
\begin{enumerate}[label=(\alph*)]
\item
    For an irreducible curve $C$ on $X$, the adjunction formula states that
    \begin{gather*}
        K_X \cdot C + C^2 + 2 - 2p_a(C) = 0
    \end{gather*}
    where $p_a(C)$ is the arithmetic genus of $C$. When $C$ is smooth, this formula relates the topological Euler characteristic $e(C) = 2 - 2p_a(C)$ to intersection numbers. When $C$ is singular, correction terms are computed as follows: The normalization $C_\norm$ of a curve is its desingularization, and the geometric genus $p_g(C)$ is defined to be the genus of the smooth curve $C_\norm$. By \eqref{inclusion-exclusion}, 
    \begin{gather*}
        e(C) = e(C_\norm) - \sum_{x \in C_\sing} (r_{x,C} - 1)
        = 2 - 2p_g(C) - \sum_{x \in C_\sing} (r_{x,C} - 1)
    \end{gather*}
    where $r_{x,C}$ is the number of branches of $C$ at $x$.

    There are classical invariants $\mu_{x,C}$ (the \emph{Milnor number}) and $\delta_{x,C}$ (the \emph{$\delta$-invariant}), associated to any plane curve singularity $x \in C_\sing$, that satisfy
    \begin{gather*}
        \mu_{x,C} = 2\delta_{x,C} - r_{x,C} + 1 \\
        p_a(C) = p_g(C) + \sum_{x \in C_\sing} \delta_{x,C}
    \end{gather*}
    (see e.g. \cite[Section 10]{milnor-2016-spc} and references therein). Then: 
    \begin{align*}
        e(C)
        = 2 - 2p_g(C) - \sum_{x \in C_\sing} (r_{x,C} - 1)
        &= 2 - 2p_a(C) + \sum_{x \in C_\sing} (2 \delta_{x,C} - r_{x,C} + 1) \\
        &= - K_X \cdot C - C^2 + \sum_{x \in C_\sing} \mu_{x,C}.
    \end{align*}
\item
    Note that $(D_i)_\sing \subset D_\sing \cap D_i$ and $\mu_{x,D_i} = 0$ if $D_i$ is smooth at $x$. Then re-index the double-sum:
    \begin{align*}
        \sum_i \Bigg[ 3 \beta_i |D_\sing \cap D_i| - 3\beta_i \sum_{x \in (D_i)_\sing} \mu_{x,D_i} \Bigg]
        &= \sum_i \sum_{x \in D_\sing \cap D_i} 3 \beta_i ( 1 - \mu_{x,D_i} ) \\
        &= \sum_{x \in D_\sing} \sum_{i: x \in D_i} 3 \beta_i ( 1 - \mu_{x,D_i} ).
    \end{align*}
    The intersection form can be computed by taking the sum of local intersection multiplicities, so
    \begin{align*}
        \sum_{i \neq j} \beta_i \beta_j D_i \cdot D_j
        &= \sum_{i \neq j} \beta_i \beta_j \sum_{x \in D_i \cap D_j} \ihat_x(D_i,D_j) \\
        &= \sum_{x \in D_\sing} \sum_{i \neq j: x \in D_i \cap D_j} \beta_i \beta_j \ihat_x(D_i,D_j).
    \end{align*}
    Note that the intersection of distinct irreducible components $D_i$ and $D_j$ is necessarily in $D_\sing$.
\end{enumerate}
In general, $\lambda_{x,D}$ depends on how the irreducible components $\{D_i\}_i$ partition the irreducible germs of $D$ meeting $x$. However, if every $D_i$ is locally irreducible at $x$ then each $D_i$ defines an irreducible germ near $x$, so the summations in \eqref{def:lambda} are sums over the irreducible germs of $D$ at $x$. Therefore, if every irreducible component $D_i$ of $D$ is locally irreducible then $\lambda_{x,D}$ depends only on the isomorphism type of $x$.
\end{proof}
\medskip

\begin{proof}[Proof of \cref{computation:intersection-numbers}]
Suppose that $C \defeq D_1 \subset X$ is an irreducible branch curve in $D$ and $\tilde C$ is an irreducible component of $\pi^{-1}(C)$. Fix the following notation:
\begin{align*}
    G_{\tilde C}
    &\defeq \text{(pointwise stabilizer of $\tilde C$)}
    = \{g \in G : \tilde C \subset \Fix(g) \} \\
    H
    &\defeq \stab(\tilde C)
    = \text{(setwise stabilizer of $\tilde C$)}
    = \{g \in G : g\tilde C = \tilde C \} \\
    b
    &\defeq |G_{\tilde C}| = \text{(branch index of $C$ with respect to $\pi$)}
\end{align*}
For the counting arguments that follow, it is convenient to use double curly-brackets $\biset{\,-\,}$ to denote a bijective set constructor, meaning a set constructor where the indexing set is in bijection with the constructed set. For example, for any $y \in Y$ the orbit of $y$ under the action of $G$ is indexed by the coset space $\sfrac{G}{G_y}$, so
\begin{gather*}
    \orb_G(y) = \{gy : g \in G \} = \biset{ gy : gG_y \in \sfrac{G}{G_y} }.
\end{gather*}

The preimage $\pi^{-1}(C)$ of $C$ is the orbit of $\tilde C$ under the action of $G$ so
\begin{gather*}
    \{ \text{ irreducible components of $\pi^{-1}(C)$ } \} = \{ g\tilde C : g \in G \} = \biset{ g\tilde C : gH \in \sfrac GH }.
\end{gather*}
For each component $g\tilde C$, the restriction $\pi|_{g\tilde C} : g\tilde C \to C$ is a finite map that is an unbranched cover away from finitely many points. The Galois group of the unbranched cover is
\begin{gather*}
    \stab(g\tilde C) / G_{g\tilde C} = gHg^{-1} / gG_{\tilde C}g^{-1}.
\end{gather*}
Therefore $\pi_* (g\tilde C) = \frac{|H|}{b} C$. Since $\pi_*\pi^*C = |G|C$, it follows that
\begin{gather*}
    \pi^*C = b \sum_{gH \in \sfrac GH} g \tilde C.
\end{gather*}

Then write $K_Y \cdot \pi^*C$ in terms of $\tilde C$:
\begin{align*}
    K_Y \cdot \pi^* C
        &= b \sum_{gH \in \sfrac GH} K_Y \cdot (g \tilde C)
        = b \sum_{gH \in \sfrac GH} (gK_Y)\cdot (g \tilde C)
        = \frac{b|G|}{|H|} K_Y \cdot \tilde C
\end{align*}
On the other hand, by \cref{lift-canonical},
\begin{align*}
    K_Y \cdot \pi^* C &= \pi^*(K_X + D) \cdot \pi^* C = |G| (K_X + D) \cdot C.
\end{align*}
Therefore,
\vspace{-0.5ex}
\begin{align*}
    K_Y \cdot \tilde C
    &= \frac{|H|}{b} \Bigg( K_X \cdot C + \sum_j \l( 1 - \frac{1}{b_j} \r) D_j \cdot C \Bigg) \\
    &= \frac{|H|}{b} \Bigg(
        K_X \cdot C + \l(1 - \frac 1b \r) C^2
        + \sum_{x \in D_\sing \cap C} \sum_{\substack{j: D_j \neq C \\ x \in D_j}} \l( 1 - \frac{1}{b_j} \r) \ihat_x(D_j, C)
    \Bigg)
\end{align*}
since $D_j \cap C \subset D_\sing \cap C$ for any irreducible component $D_j$ different from $C$. As before, $\ihat_x(C,C')$ denotes the algebraic intersection number of curves $C$ and $C'$ at a point $x$.
\medskip

To obtain an expression for $\tilde C^2$, note that $(\pi^*C)^2 = |G|C^2$ and
\begin{align*}
    (\pi^* C)^2
    &= b^2 \sum_{gH,hH \in \sfrac GH} g \tilde C \cdot h \tilde C
    = \frac{b^2|G|}{|H|} \sum_{gH \in \sfrac GH} \tilde C \cdot g \tilde C
    = \frac{b^2|G|}{|H|}
        \Bigg( \tilde C^2 + \sum_{\substack{gH \in \sfrac GH \\ gH \neq H}} \tilde C \cdot g \tilde C \Bigg)
\end{align*}
\vspace{-3.5ex}

so that
\begin{gather*}
    \tilde C^2
    = \frac{|H|}{b^2} C^2 - \sum_{\substack{gH \in \sfrac GH \\ gH \neq H}} \tilde C \cdot g \tilde C.
\end{gather*}
The last summation can be written in terms of singularity data:
\vspace{0.7ex}
\begin{align*}
    \sum_{\substack{gH \in \sfrac GH \\ gH \neq H}} \tilde C \cdot g \tilde C
    = \sum_{\substack{gH \in \sfrac GH \\ gH \neq H}} \sum_{y \in \tilde C} \ihat_y(\tilde C, g \tilde C)
    &= \sum_{x \in C} \sum_{y \in \pi^{-1}(x) \cap \tilde C} \sum_{\substack{gH \in \sfrac GH \\ gH \neq H}} \ihat_y(\tilde C, g \tilde C) \\
    &\overset{\text{(a)}}{=} \sum_{x \in C} |\pi^{-1}(x) \cap \tilde C| (\rho_{x,C} - 1) \\
    &\overset{\text{(b)}}{=} \sum_{x \in D_\sing \cap C} \frac{|H| \rho_{x,C} (\rho_{x,C} - 1)}{ |\Gamma_x| } 
\end{align*}
The labelled steps (a) and (b) are justified as follows: Fix $x \in C$ and $y \in \pi^{-1}(x) \cap \tilde C$. There exists a coordinate chart $V$ (as constructed in \cref{local-model-reflection-group}) centred at $y$ such that $G_y$ acts linearly on $V$ as a complex reflection group and the restriction $\pi|_V: V \to \pi(V)$ is the quotient of $V$ by $G_y$. The ramification locus of a complex reflection group quotient map $\pi|_V$ is the union of all of the mirrors of reflections in the reflection group $G_y$ \cite[Corollary 1.6]{steinberg-1964-dei}.

Let $C' \defeq C \cap \pi(V)$. The preimage $\pi|_V^{-1}(C')$ is a subvariety of the ramification locus of $\pi|_V$, so the irreducible components of $\pi|_V^{-1}(C')$ are mirrors. Define
\begin{align*}
    \rho_{x,C}
    &\defeq \text{number of mirrors in $\pi|_V^{-1}(C')$}
\end{align*}
Note that this number does not depend on the choice of $y$ or $V$.

\begin{enumerate}[label=(\alph*)]
\item
    If $g \in G$ and $g\tilde C \cap V \neq \emptyset$, then $g\tilde C \cap V$ is a smooth subvariety of $\pi|_V^{-1}(C')$, therefore is a mirror. All mirrors in $\pi|_V^{-1}(C')$ arise in this manner since
    \begin{gather*}
        \pi|_V^{-1}(C') = \pi^{-1}(C) \cap V = \bigcup_{gH \in \sfrac GH} (g\tilde C \cap V).
    \end{gather*}
    By analytic continuation, $(g\tilde C \cap V) = (g'\tilde C \cap V)$ if and only if $g\tilde C = g'\tilde C$. Therefore,
    \begin{align*}
        \{ \text{ mirrors in $\pi|_V^{-1}(C')$ } \}
        = \biset{ g\tilde C \cap V : gH \in \sfrac GH,\ y \in g\tilde C }.
    \end{align*}
    It then follows that
    \begin{gather*}
        \sum_{\substack{gH \in \sfrac GH \\ gH \neq H}} \ihat_y(\tilde C, g \tilde C)
        = \sum_{\substack{gH \in \sfrac GH \\ gH \neq H \\ y \in g\tilde C}} \ihat_y(\tilde C \cap V, g \tilde C \cap V)
        = \rho_{x,C} - 1
    \end{gather*}
    since two distinct mirrors intersect with intersection number 1.

\item
    Define the subset
    \begin{gather*}
        S \defeq \{ g \in G : y \in g\tilde C \} = \{ g \in G : g^{-1}y \in \pi^{-1}(x) \cap \tilde C \} \subseteq G.
    \end{gather*}
    Observe that $SH \subseteq S$ and $G_yS \subseteq S$, so $S$ is a union of left $H$-cosets, and simultaneously a union of right $G_y$-cosets. By (a),
    \vspace{-0.6ex}
    \begin{gather*}
        \rho_{x,C} = |\{ \text{ mirrors in $\pi|_V^{-1}(C')$ } \}|
        = \l| \biset{ g\tilde C \cap V : gH \in \sfrac SH } \r|
        = \frac{|S|}{|H|}.
    \end{gather*}
    On the other hand,
    \begin{gather*}
        \pi^{-1}(x) \cap \tilde C = \{g^{-1} y : g \in S \} = \biset{ g^{-1} y : G_yg \in \sfracleft{S}{G_y} }.
    \end{gather*}
    Therefore,
    \begin{gather*}
        |\pi^{-1}(x) \cap \tilde C| = \frac{|S|}{|G_y|} = \frac{|H| \rho_{x,C}}{|\Gamma_x|}.
    \end{gather*}
    Finally, note that $\rho_{x,C} - 1 = 0$ if $x \notin D_\sing \cap C$.
\end{enumerate}
\medskip

Therefore, as desired:
\begin{gather*}
    \tilde C^2
    = \frac{|H|}{b^2} C^2 - \sum_{x \in D_\sing \cap C} \frac{|H| \rho_{x,C} (\rho_{x,C} - 1)}{ |\Gamma_x| }.
\end{gather*}
By adjunction (note that irreducible ramification curves such as $\tilde C$ are smooth by \cref{finite-group-action-linearization}),
\begin{align*}
    \prop_Y(\tilde C) &= K_Y \cdot \tilde C + 3 \tilde C^2 \\
    &= \frac{|H|}{b} \xb3[ K_X \cdot C + \l( 1 + \frac 2b \r) C^2 + \sum_{x \in D_\sing \cap C} \alpha_{x,C} \xb3]
\end{align*}
where
\begin{align}
\label{def:alpha}
    \alpha_{x,C} \defeq \sum_{\substack{j: D_j \neq C \\ x \in D_j}} \l( 1 - \frac{1}{b_j} \r) \ihat_x(D_j, C) - 3\l( \frac{b \rho_{x,C} (\rho_{x,C} - 1)}{ |\Gamma_x| } \r).
\end{align}
Note that $\alpha_{x,C}$ depends only on the isomorphism type of the singularity pair $(x,C)$. 
\end{proof}

\newpage 
\subsection{Singularity coefficients}
\label{sec:singularity-coefficient-tables}

\cref{table:singularity-coefficients,table:singularity-pair-coefficients} enumerate all possible singularity coefficients for an orbifold divisor $D$, under the assumption that the irreducible components of $D$ are smooth and pairwise normal-crossing. With this assumption, any singular point $x$ on $D$ is either of type $[a,b]$ or $[a,b,c]$, and the coefficients $\lambda_{x,D}$ and $\theta_{x,D}$ depend only on the isomorphism type of $x$ by \cref{computation,rem:prop-derived}, because smooth curves are locally irreducible. The notation for singularities and singularity pairs follows \cref{singularity-notation,reflection-group-singularities}.

\begin{longtblr}[
    caption={Coefficients $\lambda_T$ and $\theta_T$ (assuming that the components of $D$ are smooth and pairwise normal-crossing).},
    label={table:singularity-coefficients},
    evaluate=\fileInput
]{ll|ll|ll}
    \SetCell[c=2]{l} \text{singularity type } T &
    & \SetCell[c=2]{l} \lambda_T &
    & \SetCell[c=2]{l} \theta_T \\
\fileInput{029-09-resources/singularity-coefficients.tex}
\end{longtblr}

\vspace{-3ex}
\begin{longtblr}[
    caption={Coefficients $\alpha_{T,C}$.},
    label={table:singularity-pair-coefficients},
    evaluate=\fileInput
]{ll|l}
    \SetCell[c=2]{l} \text{singularity pair } (T,C) && \alpha_{T,C} \\
\fileInput{029-09-resources/singularity-pair-coefficients.tex}
\end{longtblr}

\newpage 

\section{Ball quotient divisors on $\PP^2$}
\label{sec:main-theorem-proof}

The purpose of this section is to prove \cref{main-theorem}.
\smallskip

\emph{Proof outline.}\enspace
\cref{P2:DQ-DH-are-ball-quotients} gives a concise proof that $D_Q$ and $D_H$ are indeed ball quotient divisors, using Kobayashi--Nakamura--Sakai's numerical characterization of ball quotient orbifolds (\cref{ball-quotients:orbifold}).

The remainder of the work is to eliminate all other smooth pairwise normal-crossing orbifold divisors, primarily using \cref{constraints,ball-quotients:orbifold-ample}. First, \cref{P2:reduce-to-lines} reduces the problem to orbifold divisors that are line arrangements. Then \cref{P2:one-line} constrains the possible weights---a line in a ball quotient divisor with $\geq 5$ lines must have weight either 2 or 3, and weight 3 only occurs in a couple of exceptional cases. It then suffices to handle the following cases:
\begin{enumerate}
\item
line arrangements with $\geq 5$ lines where not all weights are 2, therefore featuring the exceptional cases from \cref{P2:one-line} (\cref{P2:lines-not-weight-2});
\item
line arrangements where all weights are 2 (\cref{P2:lines-weight-2}); and
\item
line arrangements with $\leq 4$ lines (\cref{P2:lines-small-cases}).
\end{enumerate}   
Then \cref{main-theorem} follows from \cref{P2:lines-weight-2,,P2:lines-not-weight-2,,P2:lines-small-cases}.
\smallskip

\emph{Notation.}\enspace
Within this section, the coefficients $\lambda_{x,D}$ and $\theta_{x,D}$ will be written as $\lambda_x$ and $\theta_x$ respectively (since $D$ is always assumed to have smooth and pairwise normal-crossing components), and $\prop_D(-)$ will be understood to mean $\prop_{(\PP^2,D)}(-)$, to reduce notational clutter. 
\smallskip

The following facts about $\PP^2$ are elementary and will be used freely without further comment: The divisor class group of $\PP^2$ is isomorphic to $\ZZ$ and generated by the hyperplane class $H$, which has self-intersection $H^2 = 1$. The canonical divisor of $\PP^2$ is linearly equivalent to $-3H$, the Euler characteristic of $\PP^2$ is $e(\PP^2) = 3$, and so $\Prop \PP^2 = 3 e(\PP^2) - K_{\PP^2}^2 = 0$.

\begin{proposition}
\label{P2:DQ-DH-are-ball-quotients}
    The following weighted line arrangements on $\PP^2$ (first named in the introduction and illustrated in \cref{fig:line-arrangements}) are ball quotient divisors:
    \begin{enumerate}[label={\normalfont(\arabic*)}]
    \item
        $D_Q$, the complete quadrilateral (6 lines, 4 triple points) with three lines of weight 2 meeting in a triple point and the three remaining lines with weight 3; and
    \item
        $D_H$, the dual Hesse arrangement (9 lines, 12 triple points) with all lines having weight 2.
    \end{enumerate}
\end{proposition}

\begin{proof}
This is a straightforward application of \cref{ball-quotients:orbifold}. All of the singularities on $D_Q$ and $D_H$ are reflection group singularities, so $D_Q$ and $D_H$ are both orbifold divisors. When $D$ is an orbifold divisor on $\PP^2$ consisting of lines, since $\Prop(\PP^2) = 0$,
\begin{align*}
    \Prop(\PP^2, D)
    &= - \sum_i \l(1 - \frac{1}{b_i} \r)^2 + \sum_{x \in D_\sing} \lambda_x.
\end{align*}
\begin{enumerate}[label=(\arabic*)]
\item
    First,
    \vspace{-4ex} 
    
    \begin{gather*}
        K_{\PP^2} + D_Q \equiv_\lin -3H + 3 \cdot \frac 23 H + 3 \cdot \frac 12 H = \frac 12 H
    \end{gather*}
    is ample. The singularity configuration on $D_Q$ is $[2,2,2]\,[2,3,3]^3\,[2,3]^3$ so, using \cref{table:singularity-coefficients},
    \begin{gather*}
        \sum_{x \in (D_Q)_\sing} \lambda_x
        = \lambda_{[2,2,2]} + 3 \lambda_{[2,3,3]} + 3 \lambda_{[2,3]}
        = \frac{3}{16} + 3 \cdot \frac{43}{144} + 3 \cdot \frac 13 = \frac{100}{48}, \\
        \Prop(\PP^2,D_Q)
        = - 3 \cdot \frac 14 - 3 \cdot \frac 49 + \frac{100}{48}
        = - \frac{36}{48} - \frac{64}{48} + \frac{100}{48} = 0.
    \end{gather*}
\item
    Observe that 
    \begin{gather*}
        K_{\PP^2} + D_H \equiv_\lin -3H + 9 \cdot \frac 12 H = \frac 32 H
    \end{gather*}
    is ample. The singularity configuration on $D_H$ is $[2,2,2]^{12}$ so
    \begin{gather*}
        \sum_{x \in (D_H)_\sing} \lambda_x
        = 12 \lambda_{[2,2,2]} = 12 \cdot \frac{3}{16} = \frac 94, \\
        \Prop(\PP^2,D_H)
        = - 9 \cdot \frac 14 + \frac 94 = 0.
    \end{gather*}
\end{enumerate}
By \cref{ball-quotients:orbifold}, both $D_Q$ and $D_H$ are ball quotient divisors on $\PP^2$.
\end{proof}

\begin{proposition}
\label{P2:reduce-to-lines}
    If $D$ is a ball quotient divisor on $\PP^2$ whose components are smooth and pairwise normal-crossing then $D$ is a line arrangement (every irreducible component of $D$ is a line).
\end{proposition}

\begin{proof}
    Let $D = \sum_i \big(1 - \frac{1}{b_i} \big)D_i$ be an orbifold divisor on $\PP^2$ whose components are smooth and pairwise normal-crossing. For each irreducible curve $D_i$, write $d_i \geq 1$ for the degree of the curve $D_i$ so that $D_i \equiv_\lin d_i H$. Compute the invariant $\sP(\PP^2,D)$ as defined in \eqref{def:prop-derived}:
    \begin{gather}
    \label{::prop-derived}
        \sP(\PP^2,D) = \frac 32 \sum_i \l( 1 - \frac{1}{b_i} \r) d_i(d_i - 1) + \sum_{x \in D_\sing} \theta_x.
    \end{gather}
    Since the components of $D$ are smooth and pairwise normal-crossing, every singularity of $D$ is of type $[a,b]$ or type $[a,b,c]$ by the classification of reflection group singularities (\cref{reflection-group-singularities}). In both cases $\theta_T = 0$. Then \eqref{::prop-derived} is a sum of non-negative terms, since $d_i(d_i - 1) \geq 0$ and $\big(1 - \frac{1}{b_i}\big) > 0$ for every $i$. Therefore \eqref{::prop-derived} vanishes if and only if $d_i = 1$ for every $i$.

    If $D$ is a ball quotient divisor then $\sP(\PP^2,D) = 0$ by \cref{constraints} so $D$ is a line arrangement.
\end{proof}

\begin{proposition}
\label{P2:one-line}
Let $D$ be an orbifold divisor on $\PP^2$ consisting of $k \geq 5$ lines. If $\prop_D(D_1) = 0$ then either $b_1 = 2$, or $b_1 = 3$ with one of the following exceptional singularity configurations: 
\begin{gather*}
\begin{array}{@{}ll}
     {[2,\mathbf{3}]\ [2,\mathbf{3},3]^2}   & (k = 6) \\
     {[2,2,\mathbf{3}]\ [2,\mathbf{3},3]^2} & (k = 7)
\end{array}
\end{gather*}
\end{proposition}

\begin{proof}
Let $q_2$ and $q_3$ denote the numbers of singularities of type $[a,b]$ and $[a,b,c]$ respectively on $D_1$. Since $D_1$ meets every other line exactly once, it follows that
\begin{gather*}
    q_2 + 2q_3 = k-1.
\end{gather*}
Let $\sT_{b_1}$ denote the collection of singularity pairs $(T,C)$ where $T$ is a triple point singularity $[a,b,c]$ and $C$ is a curve of weight $b_1$. Then define
\begin{gather*}
    \hat\alpha_3 = \min_{(T,C) \in \sT_{b_1}} \alpha_{T,C}
\end{gather*}
so that
\vspace{-0.5ex} 
\begin{gather*}
    \sum_{x \in D_\sing \cap D_1} \alpha_{x,D_1}
    \geq \frac 12 q_2 + \hat\alpha_3 q_3 = \frac 12 q_2 + \frac 12 \hat\alpha_3(k-1-q_2)
    = \frac 12 (1-\hat\alpha_3)q_2 + \frac 12 \hat\alpha_3(k-1) \\
    \prop_D(D_1) = -2 + \frac{2}{b_1} + \sum_{x \in D_\sing \cap D_1} \alpha_{x,D_1}
    \geq -2 + \frac{2}{b_1} + \frac 12 (1-\hat\alpha_3)q_2 + \frac 12 \hat\alpha_3(k-1) 
\end{gather*}
Then a sufficient condition for $\prop_D(D_1) > 0$ is if
\begin{gather*}
    k > K(b_1,q_2) \defeq \frac{4}{\hat\alpha_3} \l( 1 - \frac{1}{b_1} \r) - \l( \frac{1}{\hat\alpha_3} - 1 \r) q_2 + 1.
\end{gather*}
Note that $K(b_1,q_2)$ is decreasing with respect to $q_2$ because $0 < \hat\alpha_3 < 1$ always.

By inspecting \cref{table:singularity-pair-coefficients}, compute:
{\everymath={\displaystyle}
\begin{gather*}
\arraycolsep=2ex
\begin{array}{c c c @{\hspace{1ex}} l c @{\hspace{1ex}} l c @{\hspace{1ex}} l}
    b_1 & \hat\alpha_3 & K(b_1,0) & & K(b_1,1) & & K(b_1,2) & \\[0.8ex]
    \hline\\[-1.2ex]
    \geq 6 & 1 - \frac{3}{2b_1} & \frac{10b_1 - 11}{2b_1 - 3} & < 6 & \frac{10b_1 - 14}{2b_1 - 3} & < 6 & \frac{10b_1 - 17}{2b_1 - 3} & < 5 \\[2.5ex]
    5 & \frac{37}{60} & \frac{229}{37} & < 7 & \frac{206}{37} & < 6 & \frac{183}{37} & < 5 \\[2.5ex]
    4 & \frac{13}{24} & \frac{85}{13} & < 7 & \frac{74}{13} & < 6 & \frac{63}{13} & < 5 \\[2.5ex]
    3 & \frac{7}{20} & \frac{181}{21} & < 9 & \frac{142}{21} & < 7 & \frac{103}{21} & < 5
\end{array}
\end{gather*}
}

Note that if $k-1 = q_2 + 2q_3$ is odd then $q_2 \geq 1$. It then follows from the above computations that
\begin{itemize}
\item
    if $k \geq 6$ and $b_1 \geq 4$ then $\prop_D(D_1) > 0$;
\item
    if $k \geq 8$ and $b_1 \geq 3$ then $\prop_D(D_1) > 0$; and
\item
    if $k \geq 5$, $b_1 \geq 3$, and $q_2 \geq 2$ then $\prop_D(D_1) > 0$.
\end{itemize}

The remaining cases are handled as follows:

\underline{Case 1}:
$k = 5$, $b_1 \geq 3$, $q_2 \leq 1$. Since $k-1 = q_2 + 2q_3 = 4$ is even and $q_2 \leq 1$, it follows that $q_2 = 0$ and $q_3 = 2$. 
    Note that $\alpha_{T,C} \leq 1 - \frac{3}{2b_1}$ for any triple point $(T,C) \in \sT_{b_1}$. Therefore,
    \begin{gather*}
        \prop_D(D_1) = -2 + \frac{2}{b_1} + \alpha_{x_1,D_1} + \alpha_{x_2,D_1} \leq -2 + \frac{2}{b_1} + 2 \l( 1 - \frac{3}{2b_1} \r) = - \frac{1}{b_1} < 0.
    \end{gather*}

\underline{Case 2}: $k = 6$, $b_1 = 3$, $q_2 \leq 1$.
    Since $k-1 = q_2 + 2q_3 = 5$ is odd and $q_2 \leq 1$, it follows that $q_2 = 1$ and $q_3 = 2$.
    Let $D_2$ be the line meeting $D_1$ in a normal-crossing.
    Let $x_1$ and $x_2$ denote the triple points on $D_1$.
    For weight $b_1 = 3$, the possible triple point coefficients are:
    \begin{gather*}
        \alpha_{[2,2,\mathbf{3}]} = \frac 12,\quad
        \alpha_{[2,\mathbf{3},3]} = \frac{5}{12},\quad
        \alpha_{[2,\mathbf{3},4]} = \frac 38,\quad
        \alpha_{[2,\mathbf{3},5]} = \frac{7}{20}.
    \end{gather*}
    If $b_2 \geq 3$ then
    \begin{gather*}
        \prop_D(D_1)
        = - \frac 13 - \frac{1}{b_2} + \alpha_{x_1,D_1} + \alpha_{x_2,D_1}
        \geq - \frac 13 - \frac 13 + \frac{7}{20} + \frac{7}{20} = \frac{1}{30} > 0.
    \end{gather*}
    Therefore assume that $b_2 = 2$, so
    \begin{gather*}
        \prop_D(D_1)
        = - \frac 13 - \frac 12 + \alpha_{x_1,D_1} + \alpha_{x_2,D_1}
        = - \frac 56 + \alpha_{x_1,D_1} + \alpha_{x_2,D_1}.
    \end{gather*}
    Enumerate all possible pairs of triple points:
    \begin{align*}
        & \alpha_{[2,2,\mathbf{3}]} + \alpha_{[2,2,\mathbf{3}]} = 1
        && \alpha_{[2,\mathbf{3},3]} + \alpha_{[2,\mathbf{3},3]} = \frac 56
        && \alpha_{[2,\mathbf{3},4]} + \alpha_{[2,\mathbf{3},4]} = \frac 34 \\
        & \alpha_{[2,2,\mathbf{3}]} + \alpha_{[2,\mathbf{3},3]} = \frac{11}{12}
        && \alpha_{[2,\mathbf{3},3]} + \alpha_{[2,\mathbf{3},4]} = \frac{19}{24}
        && \alpha_{[2,\mathbf{3},4]} + \alpha_{[2,\mathbf{3},5]} = \frac{29}{40} \\
        & \alpha_{[2,2,\mathbf{3}]} + \alpha_{[2,\mathbf{3},4]} = \frac 78
        && \alpha_{[2,\mathbf{3},3]} + \alpha_{[2,\mathbf{3},5]} = \frac{23}{30}
        && \alpha_{[2,\mathbf{3},5]} + \alpha_{[2,\mathbf{3},5]} = \frac{7}{10} \\
        & \alpha_{[2,2,\mathbf{3}]} + \alpha_{[2,\mathbf{3},5]} = \frac{17}{20}
    \end{align*}
    The only configuration satisfying $\prop_D(D_1) = 0$ is
    \begin{gather*}
        [2,\mathbf{3}]\,[2,\mathbf{3},3]^2.
    \end{gather*}

\underline{Case 3}: $k = 7$, $b_1 = 3$, $q_2 \leq 1$.
    Since $k-1 = q_2 + 2q_3 = 6$ is even and $q_2 \leq 1$, it follows that $q_2 = 0$ and $q_3 = 3$.  Let $x_1$, $x_2$, and $x_3$ denote the three triple points.

    If all three singularities are different from $[2,2,\mathbf{3}]$ then
    \begin{gather*}
        \prop_D(D_1) = - \frac 43 + \alpha_{x_1,D_1} + \alpha_{x_2,D_1} + \alpha_{x_3,D_1} \leq - \frac 43 + 3 \cdot \frac{5}{12} = - \frac 43 + \frac 54 = - \frac{1}{12} < 0.
    \end{gather*}
    Thus assume that $x_1$ is of type $[2,2,\mathbf{3}]$, in which case
    \begin{gather*}
        \prop_D(D_1) = - \frac 43 + \frac 12 + \alpha_{x_2,D_1} + \alpha_{x_3,D_1} = - \frac 56 + \alpha_{x_2,D_1} + \alpha_{x_3,D_1}.
    \end{gather*}
    By referencing the finite enumeration of pairs of triple points in the previous case, the only configuration satisfying $\prop_D(D_1) = 0$ is
    \begin{gather*}
        [2,2,\mathbf{3}]\,[2,\mathbf{3},3]^2. \qedhere
    \end{gather*}
\end{proof}

\begin{proposition}
\label{P2:lines-not-weight-2}
Let $D$ be an orbifold divisor on $\PP^2$ consisting of $k \geq 5$ lines, with at least one line of weight $\geq 3$. If $\prop_D(D_i) = 0$ for all $D_i$ then $D$ is (up to projective automorphisms) the ball quotient divisor $D_Q$ from \cref{P2:DQ-DH-are-ball-quotients}.
\end{proposition}

\begin{proof}
Assume that $\prop_D(D_i) = 0$ for all $D_i$ and that there is at least one line of weight $\geq 3$. The only possible weights are 2 and 3 by \cref{P2:one-line}, so either $k = 6$ and every line of weight 3 has singularity configuration
\begin{gather*}
    [2,\mathbf{3}]\,[2,\mathbf{3},3]^2
\end{gather*}
or $k = 7$ and every line of weight 3 has singularity configuration
\begin{gather*}
    [2,2,\mathbf{3}]\,[2,\mathbf{3},3]^2.
\end{gather*}
The singularity data of one weight-3 line determines all of the weights, since every other line intersects the given weight-3 line exactly once. In both cases there are three lines of weight 3 and the remaining lines have weight 2.

\underline{Case 1}: $k = 6$. The three lines of weight $3$ form a (combinatorial) triangle with vertices the $[2,3,3]$ singularities because there are no $[3,3,3]$ singularities. Then there are exactly three $[2,3,3]$ singularities in the arrangement. The weight-2 lines incident to each of the three $[2,3,3]$ singularities are all distinct because the line between two $[2,3,3]$ vertices is a weight-3 line. There are only three weight-2 lines in the entire arrangement, so every weight-2 line is incident to a $[2,3,3]$ vertex.

It remains to determine whether the three weight-2 lines meet at a triple point or not. The singularity configuration on a weight-2 line $D_1$ must be one of 
\begin{gather*}
    [\mathbf{2},3,3]\,[\mathbf{2},3]\,[\mathbf{2},2]^2 \qquad\text{ or }\qquad
    [\mathbf{2},3,3]\,[\mathbf{2},3]\,[\mathbf{2},2,2]
\end{gather*}
because $D_1$ meets exactly 5 other lines, and all of the weight-3 lines are already accounted for in the first two singularities. In the first case,
\begin{gather*}
    \prop_D(D_1)
    = -1 + \alpha_{[\mathbf{2},3,3]} + \alpha_{[\mathbf{2},3]} + 2\alpha_{[\mathbf{2},2]}
    = -1 + \frac{1}{12} + \frac 23 + \frac 12 + \frac 12
    = \frac 34 \neq 0.
\end{gather*}
In the second case,
\begin{gather*}
    \prop_D(D_1) = -1 + \alpha_{[\mathbf{2},3,3]} + \alpha_{[\mathbf{2},3]} + \alpha_{[\mathbf{2},2,2]} = -1 + \frac{1}{12} + \frac 23 + \frac 14 = 0.
\end{gather*}
Therefore the singularity configuration of any weight-2 line must be
\begin{gather*}
    [\mathbf{2},3]\,[\mathbf{2},3,3]\,[\mathbf{2},2,2]
\end{gather*}
and the three weight-2 lines meet in a common point. Then this is a line arrangement of 6 lines with 4 triple points $[2,2,2]\, [2,3,3]^3$. Since $\PGL_3(\CC)$ acts transitively on ordered 4-tuples of points with no three collinear, it follows that this arrangement is projectively equivalent to the weighted complete quadrilateral $D_Q$.

\medskip
\underline{Case 2}: $k=7$. As in the previous case, the three lines of weight 3 form a triangle with vertices the $[2,3,3]$ singularities because there are no $[3,3,3]$ singularities, and the weight-2 lines incident to each of the three $[2,3,3]$ singularities are all distinct.

There is a unique line $D_1$ with weight 2 that does not meet any $[2,3,3]$ singularity. Then $D_1$ meets all three weight-3 lines in $[2,2,3]$ singularities. These three singularities already account for all 6 other lines, so the singularity configuration of $D_1$ is
\begin{gather*}
    [\mathbf{2},2,3]^3.
\end{gather*}
But
\begin{gather*}
    \prop_D(D_1) = -1 + 3 \alpha_{[\mathbf{2},2,3]} = -1 + 3 \cdot \frac 16 = - \frac 12 \neq 0.
\end{gather*}
Therefore, there are no arrangements of $k = 7$ lines with the desired properties.
\end{proof}

\begin{proposition}
\label{P2:lines-weight-2}
Let $D$ be an orbifold divisor on $\PP^2$ consisting of $k$ lines, with all weights $2$. If $D$ is a ball quotient divisor then $D$ is (up to projective automorphisms) the ball quotient divisor $D_H$ from \cref{P2:DQ-DH-are-ball-quotients}.
\end{proposition}

\begin{proof}
    Let $D$ be a ball quotient divisor consisting of $k$ lines, all with weight 2.
    For any line $D_i$, note that
    \begin{gather*}
        \prop_D(D_i) = -1 + \sum_{x \in D_\sing \cap D_i} \alpha_{x,D_i} = 0
    \end{gather*}
    by \cref{constraints}.
    The only possible singularities on $D$ are $[2,2]$ and $[2,2,2]$, which have coefficients
    \begin{gather*}
        \alpha_{[\mathbf{2},2,2]} = \frac 14,\qquad
        \alpha_{[\mathbf{2},2]} = \frac 12.
    \end{gather*}
    The only singularity configurations that give $\prop_D(D_i) = 0$ are:
    \begin{gather*}
    \begin{array}{@{}ll}
        {[\mathbf{2},2]^2}                   &(k = 3) \\
        {[\mathbf{2},2]\,[\mathbf{2},2,2]^2} &(k = 6) \\
        {[\mathbf{2},2,2]^4                } &(k = 9)
    \end{array}
    \end{gather*}
    Therefore, $k$ must be one of $0$, $3$, $6$, or $9$. Then $k = 9$ by \cref{ball-quotients:orbifold-ample}, since
    \begin{gather*}
        K_{\PP^2} + D \equiv_\lin -3H + k \cdot \frac 12 H
    \end{gather*}
    is ample if and only if $k \geq 7$.

    If every line has the singularity configuration $[\mathbf{2},2,2]^4$, then the only singular points of $D$ are triple points. By counting incidences, an arrangement of 9 lines with only triple points and no other singularities must have 12 triple points. The dual Hesse arrangement is the unique (up to projective equivalence) arrangement of 9 lines with 12 triple points (e.g. \cite[Theorem 1]{lampa-baczynska-wojcik-2019-dha}), so $D$ is (projectively equivalent to) the ball quotient divisor $D_H$ from \cref{P2:DQ-DH-are-ball-quotients}.
\end{proof}

\newpage 
\begin{proposition}
\label{P2:lines-small-cases}
    Let $D$ be an orbifold divisor on $\PP^2$ consisting of $k \leq 4$ lines. Then $D$ is not a ball quotient divisor.
\end{proposition}

\begin{proof}
    Let $D$ be an orbifold divisor on $\PP^2$ consisting of $k \leq 4$ lines.
    If $D$ is a ball quotient divisor then $K_{\PP^2}+D$ is ample by \cref{ball-quotients:orbifold-ample} and $\Prop(\PP^2,D) = 0$ by \cref{constraints}.

    Suppose that $K_{\PP^2} + D$ is ample. For a line arrangement,
    \begin{gather*}
        K_{\PP^2} + D
        \equiv_\lin -3H + \sum_{i=1}^k \l(1 - \frac{1}{b_i}\r) H
        = \l( k-3 - \sum_{i=1}^k \frac{1}{b_i} \r) H.
    \end{gather*}
    Evidently if $k \leq 3$ then $K_{\PP^2}+D$ is not ample. Therefore $k = 4$ and
    \begin{gather}
    \label{::k-4-ample}
        \frac{1}{b_1} + \frac{1}{b_2} + \frac{1}{b_3} + \frac{1}{b_4} < 1.
    \end{gather}
    Every singularity on $D$ is either a double point or a triple point, since $D$ is an orbifold divisor. If $D$ contains a triple point, assume that it is $[b_1,b_2,b_3]$ without loss of generality. Then
    \begin{gather*}
        \frac{1}{b_1} + \frac{1}{b_2} + \frac{1}{b_3} > 1,
    \end{gather*}
    which contradicts \eqref{::k-4-ample}. Therefore $D$ is a normal-crossing line arrangement. 

    Write $\beta_i \defeq 1 - \frac{1}{b_i}$ for brevity. Assume that $\beta_4 \geq \beta_3 \geq \beta_2 \geq \beta_1$ and let $\epsilon = \beta_4 - \beta_1 \in [0, \frac 12)$. Since $D$ is a normal-crossing line arrangement:
\begin{align*}
    \Prop(\PP^2,D)
    &= - \beta_1^2 - \beta_2^2 - \beta_3^2 - \beta_4^2 + \beta_1\beta_2 + \beta_1\beta_3 + \beta_1\beta_4 + \beta_2\beta_3 + \beta_2\beta_4 + \beta_3\beta_4 \\
    &= \beta_1(\beta_2 - \beta_1) + \beta_2(\beta_3 - \beta_2) + \beta_3(\beta_4 - \beta_3) + \beta_4(\beta_1 + \beta_2 - \beta_4) + \beta_1\beta_3 \\
    &\overset{\text{(a)}}{\geq}
    \beta_1(\beta_2 - \beta_1) + \beta_1(\beta_3 - \beta_2) + \beta_1(\beta_4 - \beta_3) + \beta_4(2\beta_1 - \beta_4) + \beta_1^2 \\
    &= \beta_1(\beta_4 - \beta_1) + \beta_4 (2\beta_1 - \beta_4) + \beta_1^2 \\
    &= \beta_1\epsilon + (\beta_1 + \epsilon)(\beta_1 - \epsilon) + \beta_1^2 \\
    &= 2 \beta_1^2 + \epsilon(\beta_1 - \epsilon)
    \overset{\text{(b)}}{\geq}
      2 \beta_1^2 > 0
\end{align*}
\begin{enumerate}[label=(\alph*)]
\item
    Note that $\beta_i \geq \beta_1$ and $\beta_{i+1} - \beta_i \geq 0$ by assumption so $\beta_i(\beta_{i+1} - \beta_i) \geq \beta_1 (\beta_{i+1} - \beta_i)$.
\item
    Note that $0 \leq \epsilon < \frac 12$ and $\beta_1 \geq \frac 12$, so $\epsilon(\beta_1 - \epsilon) \geq 0$.
\end{enumerate}

This proves that if $D$ is an orbifold divisor consisting of $k \leq 4$ lines and $K_{\PP^2}+D$ is ample, then $\Prop(\PP^2,D) \neq 0$.
\end{proof}
\medskip

\begin{proof}[Proof of \cref{main-theorem}]
  By \cref{P2:DQ-DH-are-ball-quotients}, $D_Q$ and $D_H$ are ball quotient divisors. Conversely, assume that $D$ is a ball quotient divisor on $\PP^2$ whose components are smooth and pairwise normal-crossing. By \cref{P2:reduce-to-lines}, $D$ must be a line arrangement. By \cref{P2:lines-small-cases}, $D$ has at least 5 lines. If all lines in $D$ have weight 2, then \cref{P2:lines-weight-2} states that $D$ is projectively equivalent to $D_H$. Otherwise, \cref{P2:lines-not-weight-2} states that $D$ is projectively equivalent to $D_Q$ (because $\prop_D(D_i) = 0$ for all components $D_i$ of a ball quotient divisor $D$ by \cref{constraints}).
\end{proof}

\newpage
\printbibliography[heading=bibintoc]

\end{document}